\documentclass[a4paper,reqno,12pt]{amsart}
\usepackage{amsmath,amsthm,amssymb,amscd}

\usepackage[T1]{fontenc}
\usepackage{textcomp}
\usepackage{type1cm}
\usepackage{txfonts}

\usepackage[all]{xy}
\usepackage[mathscr]{eucal}
\usepackage[dvips,usenames]{color} 
\usepackage{delarray}

\numberwithin{equation}{section}	% <-- section numbers are 
\theoremstyle{plain}

\newtheorem{theorem}{Theorem}[section]      % <-- section numbers are included
\newtheorem*{theorem*}{Theorem}

\newtheorem{lemma}[theorem]{Lemma}

\newtheorem{proposition}[theorem]{Proposition}

\newtheorem{corollary}[theorem]{Corollary}

\theoremstyle{definition}

\theoremstyle{remark}
\newtheorem{remark}[theorem]{Remark}

\newtheorem{remarks}[theorem]{Remarks}

\def\leqsl{\leqslant}

\newcommand{\trace}[1]{\ensuremath{\operatorname{tr}{\left( #1 \right)}}}

\newcommand{\tp}[1]{\ensuremath{\hspace{1.8truept}\vphantom{1}^{t}\hspace{-1pt}#1}}

\newcommand{\Mat}[1]{\ensuremath{\operatorname{Mat}_{#1}}}
   
\newcommand{\Lie}[1]{\ensuremath{\operatorname{Lie}{#1}}}
\newcommand{\Rrank}{\ensuremath{\R\textrm{-}\operatorname{rank}}}
\newcommand{\Ad}{\ensuremath{\operatorname{Ad}}}  % Adjoint
\def\H{\mathscr H}	% 

\renewcommand{\Xi}{\ensuremath{\varXi}}
\renewcommand{\Theta}{\ensuremath{\varTheta}}

\newcommand{\R}{\ensuremath{\mathbb R}}
\newcommand{\C}{\ensuremath{\mathbb C}}

\newcommand{\gl}{\ensuremath{\mathfrak {gl}}}
\renewcommand{\o}{\ensuremath{\mathfrak {o}}}
\renewcommand{\sp}{\ensuremath{\mathfrak{sp}}}
\newcommand{\g}{\ensuremath{\mathfrak{g}}}   
\newcommand{\p}{\ensuremath{\mathfrak{p}}}   % Parabolic subalgebra   
\renewcommand{\u}{\ensuremath{\mathfrak{u}}}

\newcommand{\GL}[1]{\ensuremath{\mathrm{GL}_{#1}}}
\newcommand{\rmO}{\ensuremath{\mathrm{O}}}   
\newcommand{\Ostar}{\ensuremath{\mathrm{O}^*}}   
\newcommand{\SO}{\ensuremath{\mathrm{SO}}}   
\newcommand{\U}{\ensuremath{\mathrm{U}}}   
\newcommand{\SU}{\ensuremath{\mathrm{SU}}}   
\newcommand{\Sp}{\ensuremath{\mathrm{Sp}}}   
\DeclareMathOperator{\im}{Im}		% imaginary part

\DeclareMathOperator{\rank}{rank}		% codimension
\newcommand{\dd}{\ensuremath{\operatorname{d}\!}}
\newcommand{\pd}{\ensuremath{\partial}}

\def\<{\langle}
\def\>{\rangle}

\def\c2vec#1#2{ %
   \left[ \begin{smallmatrix} % 
           #1 \\ #2  \end{smallmatrix} %
   \right]}

\newcommand{\ibar}{\ensuremath{\bar{\imath}}}
\newcommand{\jbar}{\ensuremath{\bar{\jmath}}}
\DeclareMathOperator{\impart}{Im}

\newcommand{\ai}{\ensuremath{\mathrm{i}}\,}
\renewcommand{\H}{\ensuremath{\mathbb H}}

\newcommand{\mAd}{\ensuremath{\mathbf{Ad}}}

\newcommand{\PV}{\ensuremath{\mathscr P(V)}}
\newcommand{\PVprime}{\ensuremath{\mathscr P(V')}}

\newcommand{\PVk}{\ensuremath{\mathscr P(V^k)}}
\newcommand{\PVprimek}{\ensuremath{\mathscr P(V'^k)}}

\newcommand{\PD}{\ensuremath{\mathscr{PD}}}

\title[Moment map and oscillator representation]
{The moment map on symplectic vector space \\ and oscillator representation} 

\author{Takashi Hashimoto}
\thanks{Partly supported by JSPS Grant-in-Aid for Scientific Research (C) No.~23540203 and No.~26400014.}

\address{
  University Education Center, 
  Tottori University, 
  4-101, Koyama-Minami, Tottori, 680-8550, Japan
}
\email{thashi@uec.tottori-u.ac.jp}
\date{\today}
\begin{document}

\begin{abstract}
Let $G$ denote $\Sp(n,\R), \U(p,q)$ or $\rmO^*(2n)$. 
The main aim of this paper is to show that the canonical quantization of the moment map on a symplectic $G$-vector space $(W,\omega)$ 
naturally gives rise to the oscillator (or Segal-Shale-Weil) representation of $\g := \Lie(G) \otimes \C$.
More precisely, after taking a complex Lagrangian subspace $V$ of the complexification of $W$, 
we assign an element of the Weyl algebra for $V$ to $\left< \mu, X \right>$ for each $X \in \g$, which we denote by $\left< \widehat{\mu}, X \right>$.
Then we show that the map $X \mapsto \ai \left<\widehat{\mu}, X \right>$ gives a representation of $\g$. 
With a suitable choice of $V$ in each case, the representation coincides with the oscillator representation of $\g$. 
%Taking the direct sum of $k$ copies of $W$ produces the Howe duality in the cases of the reductive dual pairs 
%$(\sp(n,\R), \rmO(k)), (\u(p,q),\U(k))$ and $(\mathfrak{o}^*(2n),\Sp(k))$ respectively.
\end{abstract}

\maketitle

\section{Introduction}

Let $(W,\omega)$ be a symplectic vector space and $\Sp(W)$ the group of linear symplectic isomorphisms of $W$. 
Then it is well known that each component of the moment map, 
i.e.,~the Hamiltonian function $H_X$ on $W$, is quadratic in the coordinate functions for any $X \in \sp(W)$ (see e.g.~\cite{CG97}).
Therefore, taking account of the fact that the commutators among the quantized operators corresponding to the coordinate functions are central,
one can see that the canonical quantization gives a representation of $\sp(W)$
since the map $X \mapsto H_X$ is a Lie algebra homomorphism from $\sp(W)$ into $C^\infty(W)$,
where the latter is regarded as a Lie algebra of infinite dimension with respect to the Poisson bracket.

The main aim of this paper is to show that for real reductive Lie groups $G=\Sp(n,\R), \U(p,q)$ and $\Ostar(2n)$,
the canonical quantization of the moment map on real symplectic $G$-vector space $(W,\omega)$ %on which $G$ acts symplectically (i.e., preserving $\omega$) 
gives rise to the oscillator (or Segal-Shale-Weil) representation of the complexified Lie algebra $\g$ of $\g_0:=\Lie(G)$ in a natural way.
Here, we understand that the canonical quantization is to construct a mapping from the space of smooth functions on $W$ 
into the ring of polynomial coefficient differential operators on a complex Lagrangian subspace $V$ of the complexification $W_\C$ of $W$,
the so-called Weyl algebra for $V$, that induces a Lie algebra homomorphism from $\g$ into the Weyl algebra.
%
%polarization (Lagrangian subspace) of $W$.
%
We remark that a different choice of a Lagrangian subspace results in a different quantization, 
and hence a different representation of the Lie algebra.
In fact, when $G=\U(p,q)$ and $\rmO^*(2n)$, 
we will find in \S\S 3-5 that one choice of a Lagrangian subspace produces finite-dimensional irreducible representations of $\g$,
while another produces infinite-dimensional ones (i.e.~the oscillator representations).

The oscillator representations have been extensively studied in relation to the Howe duality and to the minimal representations. 
Note that each Lie group $G$ we consider in this paper is a counterpart of the Howe's reductive dual pair $(G,G')$ with $G'$ compact,
i.e., $G$ and $G'$ are centralizer of each other in a symplectic group $\Sp(N,\R)$ for some $N$.
One can obtain the oscillator representations by embedding $G$ into $\Sp(N,\R)$ for $G=\U(p,q)$ and $\rmO^*(2n)$ %([\textit{ibid.}]).
(see \cite{KV78,EHW83,Howe_remarks,Howe_transcending,Howe85,DES91} etc.).

As for another approach to a construction of the oscillator representations,
we should mention \cite{HKMO12}, in which they construct the oscillator representations via Jordan algebras 
when $G$ is an arbitrary Hermitian Lie group of tube type.

It was shown in \cite{KinvDO} that for the classical Hermitian symmetric pairs 
$(G,K)=(\SU(p,q),\mathrm{S}(\U(p) \times \U(q)))$, $(\Sp(n,\R),\U(n))$, and $(\SO^*(2n),\U(n))$,
one obtains generating functions of the principal symbols of $K_\C$-invariant differential operators on $G/K$ 
in terms of determinant or Pfaffian of a certain $\g$-valued matrix whose entries are the total symbols of the differential operators 
corresponding to the holomorphic discrete series representations realized via Borel-Weil theory,
where $K_\C$ denotes a complexification of $K$.
We note that the $K_\C$-invariant differential operators play a prominent r{\^o}le in the Capelli identity (see \cite{HU91}).
Moreover, we also clarified in \cite{KinvDO} that the $\g$-valued matrix mentioned above can be regarded as the twisted moment map $\mu_\lambda$ 
on the cotangent bundle of $G/K$ which reduces to the moment map $\mu$ on the cotangent bundle when $\lambda \to 0$,
where $\lambda$ is an element of $\g^*$, the dual space of $\g$, that parametrizes the representations.
In summary, one can say that the moment map relates non-commutative objects (representation operators which are realized as differential operators) 
to commutative ones (symbols of the differential operators).
Now in this paper, we will proceed in the reverse direction: from commutative objects to non-commutative ones.

In the remainder of this section, we review a few relevant notions from symplectic geometry briefly, and state our main result.

Let $(M,\omega)$ be a real symplectic manifold.
For $f \in C^{\infty}(M)$, the space of smooth \R-valued functions on $M$, 
let $\xi_f$ denote the vector field on $M$ satisfying $\iota(\xi_f) \omega = \dd f$,
where $\iota$ stands for the contraction.
Then we define the Poisson bracket by 
\begin{equation}
\label{e:Poisson Bracket}
 \{ f, g \}:=\omega(\xi_g,\xi_f)   \qquad   (f,g \in C^{\infty}(M)),
\end{equation}
which we extend to the space of smooth $\C$-valued functions by linearity.
%
% Quantization
%
If we denote the quantum observable corresponding to a classical observable $f \in C^{\infty}(M)$ by $\widehat f$,
then the quantization principles require in particular that
\begin{equation}
\label{e:quantization principle}
 \text{if} \quad \{f_1,f_2\}=f_3 \quad \text{then} \quad [\widehat f_1, \widehat f_2]= -\ai \hbar \widehat f_3,
\end{equation}
where $\hbar$ is the Planck constant (see e.g.~\cite{BW97,Woodhouse91}); 
we set $\hbar=1$ for simplicity in what follows.

%
% moment map
%
Suppose that a  Lie group $G$ acts on $M$ symplectically, i.e., $g^* \omega=\omega$ for all $g \in G$.
A smooth map $\mu: M \to \g_0^*$ is called the moment map if the following conditions hold:
$\mu$ is $G$-equivariant, and it satisfies
\begin{equation}
\label{e:definition of moment map}
 \dd \, \< \mu, X \> =\iota(X_M) \omega \qquad \text{for all} \; X \in \g_0, 
\end{equation}
where $\g_0^*$ is the dual space of $\g_0$ and $X_M$ denotes the vector field on $M$ defined by
\begin{equation}
\label{e:vector field}
 X_M(p) = \left. \frac{\dd}{\dd t} \right|_{t=0} \exp(-t X).p    \qquad (p \in M).
\end{equation}
%
%  trace form B
%
We often identify $\g_0^*$ with $\g_0$ via the nondegenerate symmetric invariant bilinear form $B$ defined by 
\begin{equation}
\label{e:trace form}
 B(X,Y) = \begin{cases}
	    \frac12 \trace{X Y}  & \text{if} \quad \g_0=\sp(n,\R) \text{ or } \mathfrak{o}^*(2n); \\
	    \trace{X Y}          & \text{if} \quad \g_0=\u(p,q),
    	  \end{cases}
\end{equation}
which extends to the one on $\g=\sp_n$, $\mathfrak{o}_{2n}$, or $\gl_{p+q}$, the complexification of $\g_0=\sp(n,\R)$, $\mathfrak{o}^*(2n)$ or $\u(p,q)$.
If there is no risk of confusion, 
we denote the composition of $\mu$ and the isomorphism $\g_0^* \simeq \g_0$ also by $\mu$. 
Our symplectic $G$-manifold $(M,\omega)$ will be a real symplectic $G$-vector space.

%%%%%%%%%%%%%%%%%%%%%%%%%%%%%%%%%%%%%%
%
%  main result
%
%%%%%%%%%%%%%%%%%%%%%%%%%%%%%%%%%%%%%%
The main result of this paper is the following, which we prove case by case:
\begin{theorem*}
Let $G = \Sp(n,\R), \U(p,q)$ and $\rmO^*(2n)$, 
and let $(W,\omega)$ be the real symplectic $G$-vector spaces $W=\R^{2n}, (\C^{p+q})_\R$ and $(\C^{2n})_\R$ equipped with $\omega$ given by
\[
 \omega(v,w) =\begin{cases}
                \tp{v} J_n w  \quad &\text{if} \quad W=\R^{2n}, \\
                \im( v^* I_{p,q} w )  \quad &\text{if} \quad W=(\C^{p+q})_\R, \\
                \im( v^* I_{n,n} w )  \quad &\text{if} \quad W=(\C^{2n})_\R           
              \end{cases}  
\]
for $v, w \in W$, respectively,
where $J_n=\left[ \begin{smallmatrix}  & 1_n \\[3pt] -1_n &  \end{smallmatrix} \right]$
and $I_{p,q}=\left[\begin{smallmatrix} 1_p & \\[2pt] & -1_q \end{smallmatrix}\right]$.
Then, with certain choice of complex Lagrangian subspace of the complexification $W_\C$ of $W$,
the canonical quantization of the moment map $\mu: W \to \g_0^*$ yields the oscillator representations.
\end{theorem*}

%The contents of this paper
The rest of this paper is organized as follows.
In \S 2, we consider the case where $G=\Sp(n,\R)$, which is the most fundamental in this paper 
in the sense that a choice of a complex Lagrangian subspace is the key to obtain the oscillator representation. 
The original motivation of this project has started from this case with $n=1$.
In \S 3, we turn to the case where $G=\U(p,q)$ and show that 
the canonical quantization of the moment map with a certain choice of a complex Lagrangian subspace yields irreducible finite-dimensional representations of $\gl_{p+q}$,
and postpone showing that another choice leads to the oscillator representations of $\gl_{p+q}$ until \S 5.
In \S 4, we treat the case $G=\rmO^*(2n)$, in which the moment map can be expressed in two ways
due to the fact that the quaternionic vector space $\H^n$ is $\C$-isomorphic to $\C^{2n}$ and to $\Mat{n \times 2}(\C)$.
In \S 5, we take complex Lagrangian subspaces different from the ones considered in \S\S 3 and 4 in the cases of $\U(p,q)$ and $\rmO^*(2n)$:
one leading to finite-dimensional irreducible representations when $\g=\mathfrak{o}_{2n}$, and one leading to the oscillator representation when $\g=\gl_{p+q}$.
Finally, we remark a relation between the moment map and the associated variety of the corresponding irreducible $\g$-modules
occurring in the irreducible decomposition of the space of polynomials on the Lagrangian subspace under the joint action of the dual pairs $(\g,G')$.
%corresponding to the irreducible constituents occurring in the irreducible decomposition.

\vspace{10pt}

%
% Notations
%
\noindent
\textbf{Notation: }(i)
Throughout the paper, we fix a Cartan involution $\theta$ to be given by $\theta X = - X^*$.
Let $\g_0=\mathfrak k_0 \oplus \p_0$ denote the Cartan decomposition for $\g_0$,
and $\g=\mathfrak k \oplus \p$ the corresponding complexified Cartan decomposition for $\g=\g_0 \otimes \C$.

For a given basis $\{ X_\alpha\}_{\alpha}$ for $\g_0$ (resp. $\g$),
let us denote by $\{ X_\alpha^\vee \}$ its dual basis with respect to $B$, i.e., the basis for $\g_0$ (resp. $\g$) satisfying
\[
 B(X_\alpha, X_\beta^\vee) = \delta_{\alpha,\beta},
\]
where $\delta_{\alpha,\beta}$ is the Kronecker's delta, i.e., is equal to $1$ if $\alpha=\beta$ and $0$ otherwise.

\noindent
(ii)
For a positive integer $i$,
we set
\[
 \ibar := \begin{cases}
            n+i  & \text{if} \quad \g=\sp_n \text{ or } \mathfrak{o}_{2n}; \\
	    p+i  & \text{if} \quad \g=\gl_{p+q}, \\
	  \end{cases}
\]
where $\sp_n, \mathfrak{o}_{2n}$ and $\gl_{p+q}$ denote the complexified Lie algebras of $\sp(n,\R), \mathfrak{o}^*(2n)$ and $\u(p,q)$ respectively.

\section{Reductive dual pair $(\sp(n,\R),\rmO_k)$}

In this section, let $G$ denote the symplectic group $\Sp(n,\R)$ of rank $n$ over $\R$ which we realize as
\[
 \Sp(n,\R) = \{ g \in \GL{2n}(\R); \tp{g} J_n g = J_n \}
\]
with $J_n=\left[ \begin{smallmatrix}  & 1_n \\[3pt] -1_n &  \end{smallmatrix} \right]$.
Set $\g_0=\sp(n,\R)$, the Lie algebra of $G$, and take a basis for $\g_0$ as follows:
\begin{equation}
\label{e:basis for sp}
\begin{aligned}
 X_{i,j}^0 &= E_{i,j}-E_{\jbar,\ibar}       &  &(1 \leqsl i,j \leqsl n), \\
 X_{i,j}^+ &= E_{i,\jbar}+E_{j,\ibar}   &  &(1 \leqsl i \leqsl j \leqsl n),    \\
 X_{i,j}^- &= E_{\ibar,j}+E_{\jbar, i}  &  &(1 \leqsl i \leqsl j \leqsl n), 
\end{aligned}
\end{equation}
where $E_{i,j}$ denotes the matrix unit of size $2 n \times 2 n$, i.e., its $(i,j)$-th entry is $1$ and all other entries are $0$.
Note that they also form a basis for $\g=\sp_n$.

\subsection{}%{Oscillator representation}

Let $W=\R^{2 n}$ which is equipped with the canonical symplectic form $\omega$ given by
\begin{equation}
 \omega(v,w) = \tp{v} J_n w \quad (v,w \in W).
\end{equation}
Obviously, 
the natural left action of $G$ on $W$ defined by $v \mapsto g v$ (matrix multiplication) for $v \in W$ and $g \in G$ is symplectic, 
i.e., $g^* \omega=\omega$ for all $g \in G$.
If we identify the canonical base vectors $e_i:=\tp{(0,\dots,0,\overset{i\text{-th}}{1},0,\dots,0)}$ 
with $\pd_{x_i}$ for $i=1,2,\dots,n$ and with $\pd_{y_{i-n}}$ for $i=\bar 1,\bar 2, \dots, \bar n$,
then it is written as 
\begin{equation}
\label{e:symplectic form1}
 \omega =\sum_{i=1}^n \dd x_{i} \wedge \dd y_{i}
\end{equation}
at $v=\tp{(x_1,\dots,x_n, y_1,\dots,y_n)} \in W$.
%%%%%%%%%%%%%%%%%%%%%%%%%%%%%%%%%%%%%%%%%%%%%%%%%%%%%%%%%%%%%
%
%   Lemma on vector fields on W generated by sp(n,R)
%
%%%%%%%%%%%%%%%%%%%%%%%%%%%%%%%%%%%%%%%%%%%%%%%%%%%%%%%%%%%%%
\begin{lemma}
\label{l:VF sp}
The vector fields on $W $generated by the basis \eqref{e:basis for sp} for $\g_0=\sp(n,\R)$ in the sense of \eqref{e:vector field} are given by
\begin{equation}
\begin{aligned}
 (X_{i,j}^0)_W &=  -x_{j} \pd_{x_{i}} + y_{i} \pd_{y_{j}}  &       &(1 \leqsl i,j \leqsl n), \\
 (X_{i,j}^+)_W &=  - ( y_{j} \pd_{x_{i}} + y_{i} \pd_{x_{j}} ) &   &(1 \leqsl i \leqsl j \leqsl n),    \\
 (X_{i,j}^-)_W &=  - ( x_{j} \pd_{y_{i}} + x_{i} \pd_{y_{j}} ) &   &(1 \leqsl i \leqsl j \leqsl n).
\end{aligned}
\end{equation}
\end{lemma}
\begin{proof}
It is an easy exercise to show these formulae.
\end{proof}
Note that the orthogonal group $\rmO(1)=\{\pm 1\}$ also acts on $W$ symplectically on the right. 
%%%%%%%%%%%%%%%%%%%%%%%%%%%%%%%%%%%%%%%%%%%%%%
%
%     Proposition: moment map for sp_n
%
%%%%%%%%%%%%%%%%%%%%%%%%%%%%%%%%%%%%%%%%%%%%%%
\begin{proposition}
\label{p:moment map Sp}
Let $(W,\omega)$ be as above and $G=\Sp(n,\R)$.
Then the moment map $\mu:W \to \g_0^* \simeq \g_0$ is given by
\begin{equation}
\label{e:moment map1 closed}
\mu(v) = v \tp{v} J_n
       = \begin{bmatrix} - x \tp y & x \tp x \\  -y \tp y & y \tp x   \end{bmatrix}
\end{equation}
for $v= \tp{ (x_1, \dots, x_n, y_1, \dots, y_n )} \in W$ with $x=\tp{(x_1,\dots,x_n)}$ and $y=\tp{(y_1,\dots,y_n)}$.
In particular, $\mu$ is $G$-equivariant and is $\rmO(1)$-invariant.
\end{proposition}
\begin{proof}
In order to make this paper self-contained, we include the proof (see, however, e.g. \cite[Proposition 1.4.6]{CG97}).
It follows from Lemma \ref{l:VF sp} that
\begin{align*}
 \dd \, \< \mu, X_{i,j}^0 \> 
  &= \iota( (X_{i,j}^0)_W ) \omega \\
  &= \iota(-x_j \pd_{x_i}+y_i \pd_{y_j}) \sum_{k=1}^n \dd x_k \wedge \dd y_k \\
  &= -x_j \dd y_i - y_i \dd x_j = - \dd \, ( y_i x_j ).
\end{align*}
Hence one obtains that
\[
 \< \mu, X_{i,j}^0 \> = - y_i x_j.
\]
Similar calculations yield
\begin{equation*}
%\label{e:moment map1}
 \< \mu, X_{i,j}^+ \> = - y_{i} y_{j}   \quad \text{and} \quad
 \< \mu, X_{i,j}^- \> =  x_{i} x_{j}.
\end{equation*}
Therefore,
\allowdisplaybreaks{
\begin{align*}
 \mu(v)  &=\sum_{i,j} \<\mu, X_{i,j}^0 \> (X_{i,j}^0)^\vee 
          + \sum_{i \leqsl j} \< \mu, X_{i,j}^+ \> (X_{i,j}^+)^\vee 
          + \sum_{i \leqsl j} \< \mu, X_{i,j}^- \> (X_{i,j}^-)^\vee  \\
         &=\sum_{i,j} (-y_i x_j) (E_{j,i}-E_{\ibar,\jbar}) 
           + \sum_{i \leqsl j} (-y_i y_j) 2^{-\delta_{i j}} (E_{i,\jbar}+E_{j,\ibar}) 
           + \sum_{i \leqsl j } x_i x_j 2^{-\delta_{i j}} (E_{\ibar,j}+E_{\jbar,i})  \\
         &=\sum_{i,j} \left( -x_i y_j E_{i,j} + x_i x_j E_{i,\jbar} -y_i y_j E_{\ibar,j} + y_i x_j E_{\ibar,\jbar} \right) \\
         &= \begin{bmatrix}
	      -x \tp y & x \tp x \\ -y \tp y & y \tp x 
	    \end{bmatrix}
           = v \tp v J_n
\end{align*}
}%
for $v=\tp{(x_1,\dots,y_n)}$ with $x=\tp{(x_1,\dots,x_n)}$ and $y=\tp{(y_1,\dots,y_n)}$.

Now the $\rmO(1)$-invariance of $\mu$ is trivial, and the $G$-equivariance can be verified as follows:
\[
 \mu(g v) = g v \tp{(g v)} J_n = g v \tp v \tp g J_n = g v \tp v J_n g^{-1} =\Ad(g) \mu(v)
\]
since $\tp g J_n = J_n g^{-1}$ for $g \in G$.
This completes the proof.
\end{proof}

It follows from the definitions of the Poisson bracket \eqref{e:Poisson Bracket} 
and the symplectic form \eqref{e:symplectic form1} that
\begin{equation}
\label{e:Poisson bracket}
 \{ x_{i}, y_{j} \} = - \delta_{i,j}, \quad
 \{ x_i, x_j \} = \{y_i, y_j\}=0
\end{equation}
for $i,j=1,\dots,n$.
In view of \eqref{e:Poisson bracket}, we quantize the classical observables by assigning
\begin{equation}
\label{e:canonical quantization1}
 \widehat x_{i} = \mathrm{multiplication \; by \;}x_{i}, \quad
 \widehat y_{i} = - \ai \pd_{x_{i}},
\end{equation}
so that \( [\widehat x_i, \widehat y_j] = \ai  \delta_{i,j} \), as required.
In what follows, we simply denote the multiplication operator by a function $f$
by the same letter $f$ if there is no risk of confusion.

Note that the quantization \eqref{e:canonical quantization1} corresponds to taking a Lagrangian subspace of $W$ spanned by $e_1,\dots,e_n$.
However, in order to obtain a representation of the complex Lie algebra $\g=\sp_n$,
we will take a \textit{complex} Lagrangian subspace of the complexification $W_\C$ defined by
\begin{equation}
\label{e:Lagrangian1}
 V:=\<e_1,\dots,e_n\>_\C. 
\end{equation}
Therefore, the classical observables $x_j$, $j=1,\dots,n$, are now the complex coordinates on $V$ with respect to this basis.

Now, we quantize the moment map $\mu$ according to \eqref{e:canonical quantization1} and denote the quantized moment map by $\widehat \mu$ as follows:
\allowdisplaybreaks{
\begin{align}
  \widehat \mu :&= \begin{bmatrix} \, \widehat x_1 \\ \vdots \\ \, \widehat y_n \end{bmatrix}
                 \left( \widehat x_1, \dots, \widehat y_n  \right) J_n
                = \begin{bmatrix}  x  \\[3pt]  - \ai \pd_x  \end{bmatrix}
                  \left( \tp{x}, - \ai \tp{\pd_x}  \right) J_n 
          \notag
     \\
               &= \begin{bmatrix}
                  \,\ai x \tp \pd_x   & x \tp x  \\
		  \,\pd_x \tp{\pd_x}  & -\ai \pd_x \tp{x} 
                  \end{bmatrix}
          \label{e:quantized moment map1} 
\end{align}
}%
with $x=\tp{(x_1,\dots,x_n)}$ and $\pd_x=\tp{(\pd_{x_1},\dots,\pd_{x_n})}$.

Let $\PV$ denote the space of complex coefficient polynomial functions on $V$,
i.e., $\PV=\C[x_1,\dots,x_n]$,
and $\PD(V)$ the ring of polynomial coefficient differential operators on $V$.
Thus each entry of $\widehat \mu$ is an element of $\PD(V)$.

%%%%%%%%%%%%%%%%%%%%%%%%%%%%%%%%%%%%%%%%%%%%%%%%%%%%%%%%%%%%%%
%
%  Theorem: Lie algebrahomo pi Sp(n,R) \times O_1 case
%
%%%%%%%%%%%%%%%%%%%%%%%%%%%%%%%%%%%%%%%%%%%%%%%%%%%%%%%%%%%%%%
\begin{theorem}
\label{t:pi is a hom sp}
For $X \in \g=\sp_n$, set $\pi(X) = \ai \< \,\widehat \mu,  X \>$.
Then the map
\[
 \pi: \g \to \PD(V)
\] 
is a Lie algebra homomorphism.
In terms of the basis \eqref{e:basis for sp}, it is given by 
\begin{equation}
\label{e:explicit form of pi(X)}
  \pi(X) = \begin{cases}
	     -\frac12  ( x_i \pd_{x_j} + \pd_{x_j} x_i)  &\text{if} \quad X=X_{i,j}^0 ; \\[3pt]
             \; \ai  \pd_{x_i} \pd_{x_j}  &\text{if} \quad X=X_{i,j}^+ ; \\[3pt]
             \; \ai  x_{i} x_{j}   &\text{if} \quad X=X_{i,j}^-.
           \end{cases}
\end{equation}
\end{theorem}
\begin{proof}
Of course, one can verify the commutation relations among the explicit form \eqref{e:explicit form of pi(X)},
which can be easily deduced from \eqref{e:quantized moment map1},
coincides with those of the basis $\{X_{i,j}^\star \}$ for $\g$.
However, we will give another proof in the following.

The moment map $\mu$ induces a Lie algebra homomorphism from $\g_0$ to $C^{\infty}(W)$, i.e.,
if we write $H_X:=\< \mu, X \>$ for $X \in \g_0$, then we have
\begin{equation}
\label{e:hamiltonian}
 \{ H_X, H_Y \} = H_{[X,Y]} \qquad (X,Y \in \g_0).
\end{equation}
Taking account of the facts that both Poisson bracket and commutator are derivations,
one sees that the relation \eqref{e:hamiltonian} implies that
\[
 [\widehat H_X, \widehat H_Y] = -\ai \widehat H_{[X,Y]}
\]
as required in \eqref{e:quantization principle}, 
since each function $H_X$ is quadratic in the coordinate functions $x_i, y_j$ for any $X \in \g_0$ (see \cite{CG97})
and the commutators among $\widehat x_i$ and $\widehat y_j$ are in the center of $\PD(V)$ for $i,j=1,\dots,n$.
Hence, it follows from $\pi(X)= \ai \widehat H_X$ that
\[
 [ \pi(X), \pi(Y) ] = \pi([X,Y]) \quad (X,Y \in \g_0). 
\]
Now, extend the result to the complexification by linearity.
%and we are done.
\end{proof}

%%%%%%%%%%%%%%%%%%%%%%%%%%%
%
%     Motivating remark
%
%%%%%%%%%%%%%%%%%%%%%%%%%%%
\begin{remark}
By \eqref{e:quantized moment map1}, 
one can rewrite $\pi(X)=\ai \<\widehat \mu, X\>$, $X \in \g$, as follows:
\allowdisplaybreaks{
\begin{align*}
 \pi(X) 
   &= \frac{\ai}2 \trace{\, \widehat \mu \, X }
   = \frac{\ai}2 \trace{
                    \begin{bmatrix}  x  \\[3pt]  - \ai \pd_x  \end{bmatrix}
                    \left( \tp{x}, - \ai \tp{\pd_x}  \right) J_n X
                 }
     \\
   &= \frac{\ai}2  \left( \ai \tp{\pd_x}, \tp{x} \right) X
                    \begin{bmatrix}  x  \\[3pt]  - \ai \pd_x  \end{bmatrix},
\end{align*}
}%
where the last equality follows from the fact that $X$ is a member of $\g$.
Namely, our quantized moment map $\widehat \mu$ is essentially identical to the map $\varphi$ %(in their notation)
given in Example of \cite[Chap.~I, \S 6]{Knapp-Vogan95}.
This observation was the original motivation of the present work.
\end{remark}

It is well known that the irreducible decomposition of the representation $(\pi,\PV)$ of $\g$ is given by
\( \PV = \PV_{+} \oplus \PV_{-} \), 
where $\PV_{+}$ and $\PV_{-}$ are the subspaces consisting of even polynomials $f(x)$ satisfying $f(-x)=f(x)$ 
and of odd polynomials $f(x)$ satisfying $f(-x)=-f(x)$ respectively.
It is also well known that this phenomena can be explained by the type of representations of $\rmO(1)$ which acts on $V$ on the right.

\subsection{}%{Reductive dual pair}

Let us consider the vector space $W^k := W \oplus \dots \oplus W$, the direct sum of $k$ copies of $W=\R^{2n}$, 
which can be identified with $\Mat{2n \times k}(\R)$.
It is a symplectic vector space equipped with symplectic form $\omega_k$ given by
\[
 \omega_k (v,w) =\trace{\tp v J_n w} \qquad ( v,w \in W^k ).
\]
Let $e_{i,a}$ denote the matrix unit of size $2n \times k$ for $i=1,\dots,n$ and $a=1,\dots,k$.
Under the identification $e_{i,a} \leftrightarrow \pd_{x_{i,a}}$ and $e_{\ibar,a} \leftrightarrow \pd_{y_{i,a}}$,
we write $v={}^t[ x_1, \dots,  x_n,  y_1, \dots,  y_n ] \in W^k$ 
with $x_i=(x_{i,1},\dots,x_{i,k})$ and $y_i=(y_{i,1},\dots,y_{i,k})$ being row vectors%
%%%%%% footnote %%%%%
\footnote{%
More precisely, one should write an element $v \in W^k=\Mat{2n \times k}(\R)$ as $v={}^t[ \tp x_1, \dots,  \tp x_n,  \tp y_1, \dots,  \tp y_n ]$; 
however, we will adopt this abbreviated notation in what follows. 
      }
%%%%%%%%%%%%%%%%%%%%% 
of size $k$ for $i=1,\dots,n$.
Then $\omega_k$ is given by
\begin{equation}
\label{e:symplectic form1.2}
 \omega_k 
%     = \sum_{i=1}^n \dd x_i \wedge \dd \! \tp y_i 
     = \sum_{1 \leqsl i \leqsl n, 1 \leqsl a \leqsl k} \dd x_{i,a} \wedge \dd y_{i,a}
\end{equation}
at $v={}^t[ x_1,\dots, y_n \,]$.
Note that $G=\Sp(n,\R)$ acts on $W^k=\Mat{2n \times k}(\R)$ on the left, 
while the real orthogonal group $\rmO(k)$ acts on the right.
Both actions are symplectic.

For brevity, let us write $x_\star \cdot y_\star = \sum_{a=1}^k x_{\star,a} \, y_{\star,a}$, the standard inner product between two row vectors 
$x_\star=(x_{\star,1},\dots,x_{\star,k})$ and $y_\star=(y_{\star,1},\dots,y_{\star,k})$ of size $k$ in what follows. 
%%%%%%%%%%%%%%%%%%%%%%%%%%%%%%%%%%%%%%%%%%%%%%%%%%%%%%%%%%%%%%
%
%   Propostion: moment map in the case of Sp(n,R) times O(k)
%
%%%%%%%%%%%%%%%%%%%%%%%%%%%%%%%%%%%%%%%%%%%%%%%%%%%%%%%%%%%%%%
\begin{proposition}
\label{p:moment map (Sp,O)}
Let $(W^k,\omega_k)$ be the symplectic $G$-vector space.
Then the moment map $\mu:W^k \to \g_0^* \simeq \g_0$ is given by
\allowdisplaybreaks{
\begin{subequations}
  \begin{align}
 \mu( v ) 
   &= v \tp{v} J_n
          \\
   &= \left[
 \begin{array}{cccc|cccc} 
 -x_1 \cdot {y_1} & -x_1 \cdot {y_2}  & \cdots  &  -x_1 \cdot {y_n}   & x_1 \cdot {x_1} & x_1 \cdot {x_2} & \cdots & x_1 \cdot {x_n}         
    \\
 -x_2 \cdot {y_1} & -x_2 \cdot {y_2}  & \cdots  &  -x_2 \cdot {y_n}   & x_2 \cdot {x_1} & x_2 \cdot {x_2} & \cdots & x_2 \cdot {x_n}         
    \\  
 \vdots   & \vdots    & \ddots  & \vdots      & \vdots  & \vdots  & \ddots  & \vdots
    \\[2pt]
 -x_n \cdot {y_1} & -x_n \cdot {y_2}  & \cdots  &  -x_n \cdot {y_n}   & x_n \cdot {x_1} & x_n \cdot {x_2} & \cdots & x_n \cdot {x_n}
    \\[3pt]  \hline 
 -y_1 \cdot {y_1} & -y_1 \cdot {y_2}  & \cdots  &  -y_1 \cdot {y_n}   & y_1 \cdot {x_1} & y_1 \cdot {x_2} & \cdots & y_1 \cdot {x_n}
    \\
 -y_2 \cdot {y_1} & -y_2 \cdot {y_2}  & \cdots  &  -y_2 \cdot {y_n}   & y_2 \cdot {x_1} & y_2 \cdot {x_2} & \cdots & y_2 \cdot {x_n}         
    \\  
 \vdots   & \vdots    & \ddots  & \vdots      & \vdots  & \vdots  & \ddots  & \vdots
    \\
 -y_n \cdot {y_1} & -y_n \cdot {y_2}  & \cdots  &  -y_n \cdot {y_n}   & y_n \cdot {x_1} & y_n \cdot {x_2} & \cdots & y_n \cdot {x_n}
 \end{array}    
   \right]
 \end{align}
\end{subequations}
}%
for $v={}^t[ x_1, \dots,  x_n,  y_1, \dots,  y_n ] \in W^k$.
In particular, $\mu$ is $G$-equivariant and is $\rmO(k)$-invariant.
\end{proposition}
\begin{proof}
In the agreements that $x_i$, $y_i$, $\pd_{x_i}$ and $\pd_{y_i}$ denote row vectors 
and that the products stand for the inner product of row vectors,
a simple calculation shows that
the vector fields on $W^k $generated by the basis \eqref{e:basis for sp} for $\g_0=\sp(n,\R)$ are given by the same formulae as in Lemma \ref{l:VF sp},
and thus the same argument given in the proof of Proposition \ref{p:moment map Sp} produces the result. 
\end{proof}

It follows from \eqref{e:symplectic form1.2} that the Poisson brackets among the coordinate functions 
$x _{i,a},y_{i,a}$, ${i=1,\dots,n;a=1,\dots,k}$, are given by
\[
  \{ x_{i,a}, y_{j,b} \} = \delta_{i,j} \delta_{a,b}
\]
and all other brackets vanish.
Therefore, we quantize them by assigning
\[
  \widehat x_{{i,a}} = x_{i,a} \quad \text{and} \quad \widehat y_{i,a} = -\ai \pd_{x_{i,a}}
\]
for $i=1,\dots,n$ and $a=1,\dots,k$.

Let $V^k$ denote the direct sum $V \oplus \dots \oplus V$ ($k$ copies) with $V$ given in \eqref{e:Lagrangian1}.
Since $V^k$ can be identified with $\Mat{n \times k}(\C)$, the upper half of $W_\C^k=\Mat{2n \times k}(\C)$,
we write an element of $V^k$ as $x=(x_{i,a})_{i=1,\dots,n;a=1,\dots,k}$.
Let $\PVk = \C[x_{i,a} ; i=1,\dots,n, a=1,\dots,k]$ be the algebra of complex polynomial functions on $V^k$,
and $\PD(V^k)$ the ring of polynomial coefficient differential operators on $V^k$.
Note that $x_{i,a}$'s are now complex variables 
and that the complex general linear group $\GL{k}$ acts on $V^k$ by matrix multiplication on the right, and thus on $\PVk$ by right translation:
\begin{equation}
\label{e:right regular rep. of G'}
 \rho(g)f (x) := f(x g) \quad ( g \in \GL{k}, f \in \PVk ).
\end{equation}
The right-action of $\GL{k}$ on $V^k$ is the restriction of the one on $W_\C^k$.

The quantized moment map $\widehat \mu$ in this case is also given by the same formula as \eqref{e:quantized moment map1}:
\begin{equation*}
\widehat \mu 
 = \begin{bmatrix}
    \,\ai x \tp \pd_x   & x \tp x  \\
    \,\pd_x \tp{\pd_x}  & -\ai \pd_x \tp{x} 
   \end{bmatrix}.
\end{equation*}
In this case, however, $x$ and $\pd_x$ are $n \times k$-matrices 
whose $(i,a)$-th entries are the multiplication operator $x_{i,a}$ and the differential operator $\pd_{x_{i,a}}$ 
for $i=1,\dots,n$ and $a=1,\dots,k$, respectively.

%%%%%%%%%%%%%%%%%%%%%%%%%%%%%%%%%%%%%%%%%%%%%%%%%%%%%%%%%%%%%%
%
%     Lemma:  
%     Relation between rho and differentiatio or multiplication operator
%
%%%%%%%%%%%%%%%%%%%%%%%%%%%%%%%%%%%%%%%%%%%%%%%%%%%%%%%%%%%%%%
\begin{lemma}
\label{l:adjoint of rho(G')}
For $x=(x_{i,a})_{i=1,\dots,n;a=1,\dots,k} \in V^k$ and $ g \in \GL{k}$,
the following relations hold in $\operatorname{End}\left( \PVk \right)$:
\allowdisplaybreaks{
\begin{align}
 \rho(g)^{-1} \pd_{x_{i,a}} \, \rho(g) &= \sum_{b} g_{a b} \pd_{x_{i,b}},
     \label{e:Adjoint on pd}
     \\
 \rho(g)^{-1} x_{i,a} \, \rho(g) &= \sum_{b} g^{b a} x_{i,b}, 
     \label{e:Adjoint on x}
\end{align}
}%
where $g=(g_{a b})$ and $g^{-1}=(g^{a b})$. 
\end{lemma}
\begin{proof}
Since $\pd_{x_{i,a}}$ is identified with $e_{i,a} \in \Mat{n \times k}(\C)$, one sees that
\begin{align*}
 \left( \pd_{i,a} (\rho(g) f) \right) (x) 
     &= \left. \frac{\dd}{\dd t} \right|_{t=0} f(x g + t e_{{i,a}} g ) 
     = \sum_{b=0}^k g_{a b} \frac{\pd f}{\pd x_{i,b}} (x g),
\end{align*}
and hence
\[
 \left( \rho(g)^{-1} \pd_{x_{i,a}} \,\rho(g) \right) f = \sum_{b=1}^k g_{a b} \frac{\pd f}{\pd x_{i,b}}
\]
for $f \in \PVk$.
Thus one obtains \eqref{e:Adjoint on pd}.

On the other hand, since
\[
 \left( \rho(g)^{-1}( x_{i,a} f) \right) (x) = \biggl( \sum_{b=1}^k x_{i,b} g^{b a} \biggr) \, f(x g^{-1})
\]
one has
\[
 \left( \rho(g)^{-1} x_{i,a} \, \rho(g) \right) f = \biggl( \sum_{b=1}^k g^{b a} x_{i,b} \biggr) \, f
\]
and \eqref{e:Adjoint on x}.
\end{proof}
Let us abbreviate as $\rho(g) \, a \, \rho(g)^{-1}=:\Ad_{\rho(g)} a$ for $a \in \PD(V^k)$ and $g \in \GL{k}$.
Moreover, for a given matrix $A=(a_{i j})$ with $a_{i j} \in \PD(V^k)$,
let us denote by $\mAd_{\rho(g)} A=(\Ad_{\rho(g)}a_{i j})$, 
the matrix whose $(i,j)$-th entries are equal to $\Ad_{\rho(g)} a_{i j}$.

%%%%%%%%%%%%%%%%%%%%%%%%%%%%%%%%%%%%%%%%%%%%%%%%%%%%%%%%%%%%%%
%
%
%     Theorem: Lie algebra homo pi Sp(n)-O_k case
%
%
%%%%%%%%%%%%%%%%%%%%%%%%%%%%%%%%%%%%%%%%%%%%%%%%%%%%%%%%%%%%%%
\begin{corollary}
\label{t:pi is a hom (sp,O)}
For $X \in \g=\sp_n$, set $\pi(X) = \ai \< \,\widehat \mu,  X \>$.
Then the map 
\[
 \pi:\g \to \PD(V^k)
\] 
is a Lie algebra homomorphism.
In terms of the basis \eqref{e:basis for sp}, it is given by 
\begin{equation}
\label{e:explicit form of pi (sp,O)}
  \pi(X) = \begin{cases}
	     -\frac12  \sum_{a=1}^k ( x_{i,a} \pd_{x_{j,a}} + \pd_{x_{j,a}} x_{i,a})  &\text{if} \quad X=X_{i,j}^0 ; \\[3pt]
             \; \ai  \sum_{a=1}^k \pd_{x_{i,a}} \pd_{x_{j,a}}  &\text{if} \quad X=X_{i,j}^+ ; \\[3pt]
             \; \ai  \sum_{a=1}^k x_{i,a} x_{j,a}   &\text{if} \quad X=X_{i,j}^-.
           \end{cases}
\end{equation}
Moreover, $\pi(X)$ commutes with the action of the complex orthogonal group%
%%%%% footnote %%%%%
\footnote{%
We realize the complex orthogonal group as $\rmO_k=\{ g \in \GL{k}; \tp g g = 1_k \}$ in this section.
          }
%%%%%%%%%%%%%%%%%%%%
$\rmO_k$, i.e., $\pi(X) \in \PD(V^k)^{\rmO_k}$ for all $X \in \sp_n$.
\end{corollary}
\begin{proof}
The same argument as in the proof of Theorem \ref{t:pi is a hom sp} shows that $\pi : \g \to \PD(V^k)$ is a Lie algebra homomorphism
and that \eqref{e:explicit form of pi (sp,O)} holds.

For the last statement, it follows from Lemma \ref{l:adjoint of rho(G')} that
\begin{equation*}
 \mAd_{\rho(g)^{-1}} x_i = x_i \, g^{-1}
     \quad \text{and} \quad
 \mAd_{\rho(g)^{-1}} \pd_{x_i} = \pd_{x_i} \tp g
\end{equation*}
with $x_i=(x_{i,1},\dots,x_{i,k})$ and $\pd_{x_i}=(\pd_{x_{i,1}},\dots,\pd_{x_{i,k}})$ for $g \in \GL{k}$.
Hence, if $g \in \rmO_k$ then one has
\[
  \begin{bmatrix} \mAd_{\rho(g)^{-1}} x\,  \\[3pt] -\ai \mAd_{\rho(g)^{-1}} \pd \end{bmatrix}
% = \begin{bmatrix} x \, g^{-1} \\ \pd \tp g \end{bmatrix} 
 = \begin{bmatrix} x\,  \\[3pt] -\ai \pd  \end{bmatrix}  \tp g
\]
with $x=\!{}^t [x_1,\dots,x_n]$ and $\pd={\!{}^t} [\pd_{x_1},\dots,\pd_{x_n}]$,
since $\tp g = g^{-1}$.
Therefore,
\allowdisplaybreaks{
\begin{align*}
\mAd_{\rho(g)^{-1}} \widehat \mu 
  &=\begin{bmatrix}
      \ai \mAd_{\rho(g)^{-1}}(x \tp \pd_x)   & \mAd_{\rho(g)^{-1}}(x \tp x)  \\[3pt]
      \mAd_{\rho(g)^{-1}}(\pd_x \tp{\pd_x})  & -\ai \mAd_{\rho(g)^{-1}}(\pd_x \tp{x}) 
    \end{bmatrix} 
           \\
  &=\begin{bmatrix}
     \ai \mAd_{\rho(g)^{-1}} x \, \tp{(\mAd_{\rho(g)^{-1}} \pd_x)} & \mAd_{\rho(g)^{-1}} x \, \tp{(\mAd_{\rho(g)^{-1}} x)} \\[3pt]
     \mAd_{\rho(g)^{-1}} \pd_x \, \tp{(\mAd_{\rho(g)^{-1}} \pd_x)} & -\ai \mAd_{\rho(g)^{-1}} \pd_x \, \tp{(\mAd_{\rho(g)^{-1}} x)}
    \end{bmatrix}
           \\
  &=\begin{bmatrix} \mAd_{\rho(g)^{-1}} x \\[3pt] -\ai \mAd_{\rho(g)^{-1}} \pd_x \end{bmatrix}
    \begin{bmatrix} \tp{(\mAd_{\rho(g)^{-1}} x)}, & -\ai \tp{(\mAd_{\rho(g)^{-1}} \pd_x)} \end{bmatrix} J_n
           \\
  &=\begin{bmatrix} x\,  \\[3pt] -\ai \pd  \end{bmatrix}  \tp g
     g \begin{bmatrix} \tp x,  & -\ai \tp \pd  \end{bmatrix} J_n
     = \widehat \mu.
\end{align*}
}%
This completes the proof.
\end{proof}
It is well known that the irreducible decomposition of $\PVk$ under the joint action of $(\sp_n, \rmO_k)$ is given by
\begin{equation}
\label{e:Howe duality1}
% \PVk \simeq \sum_{\substack{\sigma \in \widehat{\rmO}_k \\ L(\sigma) \ne \{ 0 \} }} L(\sigma) \otimes V_\sigma,
 \PVk \simeq \sum_{{\sigma \in \widehat{\rmO}_k, \, L(\sigma) \ne \{ 0 \} }} L(\sigma) \otimes V_\sigma,
\end{equation}
where $V_\sigma$ is a representative of the class $\sigma \in \widehat{\rmO}_k$,
the set of all equivalence classes of the finite-dimensional irreducible representation of $\rmO_k$,
and $L(\sigma):=\operatorname{Hom}_{\,\mathrm{O}_k}(V_\sigma, \PVk)$ which is an infinite-dimensional irreducible representation of $\sp_n$.
Moreover, it is also known that the action $\pi$ restricted to $\mathfrak k$ lifts to the double cover $\Tilde K_\C$
of the complexification $K_\C$ of the maximal compact subgroup $K$ of $G=\Sp(n,\R)$, 
which implies that $L(\sigma)$ is an irreducible $(\g,\Tilde K_\C)$-module.

Note that our realization of the representation $\pi$ is the Schr\"{o}dinger model of the oscillator representation of $\g=\sp_n$;
we will need another realization of the representation in \S 5, i.e., the \emph{Fock model}.

\section{Reductive dual pair $(\u(p,q),\GL{k})$}
\label{s:GL-GL}

Let $G$ denote the indefinite unitary group defined by
\[
 \U(p,q) = \{ g \in \GL{n}(\C); g^* I_{p,q} g = I_{p,q} \}
\]
with $I_{p,q}=\left[\begin{smallmatrix} 1_p & \\[2pt] & -1_q \end{smallmatrix}\right]$,
and put $n=p+q$ only in this section for brevity.
Set $\g_0=\u(p,q)$, the Lie algebra of $G$, and take a basis for $\g_0$ as follows:
\allowdisplaybreaks{
\begin{equation}
\label{e:basis for u(p,q)}
\begin{aligned}
 X_{i,j}^c &= E_{i,j}-E_{j,i}                &   &(1 \leqsl i < j \leqsl p, \textrm{ or } p+1 \leqsl i < j \leqsl n), \\
 Y_{i,j}^c &= \ai (E_{i,j}+E_{j,i})    &   &(1 \leqsl i \leqsl j \leqsl p, \textrm{ or } p+1 \leqsl i \leqsl j \leqsl n), \\
 X_{i,j}^n &= E_{i, \jbar}+E_{\jbar, i}          &   &(1 \leqsl i \leqsl p, 1 \leqsl j  \leqsl q), \\
 Y_{i,j}^n &= \ai ( E_{i, \jbar} - E_{\jbar, i}) &   &(1 \leqsl i \leqsl p, 1 \leqsl j  \leqsl q),
\end{aligned}
\end{equation}
}%
where $E_{i,j}$ denotes the matrix unit of size $n \times n$.
Note that $E_{i,j}$, $i,j=1,\dots,n$, form a basis for $\g=\gl_{n}$, the complexified Lie algebra of $\g_0=\u(p,q)$.

\subsection{}

Let $W=(\C^{n})_\R$, the underlying real vector space of the complex vector space $\C^{n}$,
and $H: \C^{n} \times \C^{n} \to \C$ the indefinite Hermitian form given by
\[
 H(z,w):= z^* I_{p,q} w  \quad (z,w \in \C^n).
\]
We regard $W$ as symplectic manifold with symplectic form $\omega = \impart H$,
where $\impart H$ stands for the imaginary part of $H$.
Under the identification $e_j \leftrightarrow \pd_{x_j}$ and $\ai e_j \leftrightarrow \pd_{y_j}$ for $j=1,\dots,n$,
it is explicitly given by
\begin{equation}
\label{e:symplectic form2}
 \omega=\sum_{j=1}^{n} \epsilon_j \dd x_j \wedge \dd y_j
\end{equation}
at $z=x+\ai y \in W$ with $x=\tp{(x_1,\dots,x_{n})},\, y=\tp{(y_1,\dots,y_{n})} \in \R^{n}$,
where 
\begin{equation}
 \epsilon_j:=\begin{cases} \hphantom{-} 1 & (j=1,\dots,p) \\ -1 &(j=p+1,\dots,n). \end{cases}
\end{equation}
Then $(W,\omega)$ is a symplectic $G$-manifold since the natural action of $G$ on $\C^{n}$ preserves the Hermitian form $H$.
\begin{lemma}
\label{l:VF unitary}
The vector fields on $W$ generated by the basis \eqref{e:basis for u(p,q)} for $\g_0=\u(p,q)$ in the sense of \eqref{e:vector field} are given by
\begin{equation}
\label{e:VF unitary}
\begin{aligned}
 (X_{i,j}^c)_W &= - x_{j} \pd_{x_{i}} - y_{j} \pd_{y_{i}} + x_{i} \pd_{x_{j}} + y_{i} \pd_{y_{j}},  \\
 (Y_{i,j}^c)_W &= y_{j} \pd_{x_{i}} - x_{j} \pd_{y_{i}} + y_{i} \pd_{x_{j}} - x_{i} \pd_{y_{j}}, \\
 (X_{i,j}^n)_W &= - x_{\jbar} \pd_{x_{i}} - x_{i} \pd_{x_{\jbar}} - y_{\jbar} \pd_{y_{i}} - y_{i} \pd_{y_{\jbar}},    \\
 (Y_{i,j}^n)_W &= y_{\jbar} \pd_{x_{i}} - x_{\jbar} \pd_{y_{i}} - y_{i} \pd_{x_{\jbar}} + x_{i} \pd_{y_{\jbar}}.
\end{aligned}
\end{equation}
\end{lemma}

Note that the unitary group $\U(1)$ also acts on $W$ symplectically on the right.
%%%%%%%%%%%%%%%%%%%%%%%%%%%%%%%%%%%%%%%%%%%%%%%%%%%%%%%%%%%%%%
%
%     Moment map of U(p,q) case
%
%%%%%%%%%%%%%%%%%%%%%%%%%%%%%%%%%%%%%%%%%%%%%%%%%%%%%%%%%%%%%%
\begin{proposition}
\label{p:moment map2}
Let $(W, \omega)$ be as above and $G=\U(p,q)$.
Then the moment map $\mu: W \to \g_0^* \simeq \g_0$ is given by
\begin{equation}
\label{e:moment map2 closed} 
\mu(z) = - \frac{\ai}2 z z^* I_{p,q}
\end{equation}
for $z=x + \ai y \in W$ with $x=\tp{(x_1,\dots,x_{n})}, y=\tp{(y_1,\dots,y_{n})} \in \R^{n}$.
In particular, $\mu$ is $G$-equivariant and is $\U(1)$-invariant.
\end{proposition}
\begin{proof}
It follows from Lemma \ref{l:VF unitary} that
\allowdisplaybreaks{
\begin{equation}
\label{e:moment map2 real}
 \< \mu, X \>=\begin{cases}
               \epsilon_{i} (x_{i} y_{j} - x_{j} y_{i})  &\text{if} \quad X=X_{i,j}^c ; \\
	       \epsilon_i (x_{i} x_{j} + y_{i} y_{j})    &\text{if} \quad X=Y_{i,j}^c ; \\
                x_{i} y_{\jbar} - x_{\jbar} y_{i}            &\text{if} \quad X=X_{i,j}^n ; \\
                x_{i} x_{\jbar} + y_{i} y_{\jbar}             &\text{if} \quad X=Y_{i,j}^n,
	      \end{cases} 
\end{equation}
which can be rewritten in terms of the complex coordinates defined by $z_j = x_j + \ai y_j$ $(j=1,\dots,n)$ and their complex conjugates as
\begin{equation}
\label{e:moment map2 complex}
 \< \mu, X \>=\begin{cases}
               \frac{\ai}2 \epsilon_{i} (z_{i} \bar z_{j} - z_{j} \bar z_{i})  &\text{if} \quad X=X_{i,j}^c  ; \\
	       \frac12 \epsilon_i (z_{i} \bar z_{j} + z_{j} \bar z_{i})    &\text{if} \quad X=Y_{i,j}^c      ; \\
               \frac{\ai}2 (z_{i} \bar z_{\jbar} - z_{\jbar} \bar z_{i})            &\text{if} \quad X=X_{i,j}^n ; \\
               \frac12 (z_{i} \bar z_{\jbar} + z_{\jbar} \bar z_{i})            &\text{if} \quad X=Y_{i,j}^n.
	      \end{cases} 
\end{equation}
Hence,
\begin{align*}
 \mu(z)  &=\sum_{i < j} \<\mu, X_{i,j}^c \> (X_{i,j}^c)^\vee 
          + \sum_{i \leqsl j} \< \mu, Y_{i,j}^c \> (Y_{i,j}^c)^\vee 
          + \sum_{i,j} \< \mu, X_{i,j}^n \> (X_{i,j}^n)^\vee  
          + \sum_{i,j} \< \mu, Y_{i,j}^n \> (Y_{i,j}^n)^\vee  \\
         &= -\frac{\ai}2 \sum_{1 \leqsl i,j \leqsl p} z_i \bar z_j E_{i,j}
           + \frac{\ai}2 \sum_{1 \leqsl i,j \leqsl q} z_{\ibar} \bar z_{\jbar} E_{\ibar,\jbar}
           + \frac{\ai}2 \sum_{1 \leqsl i \leqsl p, 1 \leqsl j \leqsl q} z_i \bar z_{\jbar} E_{i,\jbar}
           - \frac{\ai}2 \sum_{1 \leqsl i \leqsl q, 1 \leqsl j \leqsl p} z_{\ibar} \bar z_j E_{\ibar,j} \\
         &= -\frac{\ai}2 z z^* I_{p,q},
\end{align*}
}%
with $z=\tp{(z_1,\dots,z_n)}$.

The $\U(1)$-invariance of $\mu$ is obvious, and the $G$-equivariance can be verified as follows:
\[
 \mu(g z) = -\frac{\ai}2 (g z) (g z)^* I_{p,q} = -\frac{\ai}2 g z z^* g^* I_{p,q} = \Ad(g) \mu(z)
\]
since $g^* I_{p,q} = I_{p,q} g^{-1}$ for $g \in \U(p,q)$.
\end{proof}

It follows from \eqref{e:symplectic form2} that 
the Poisson brackets among the real coordinate functions $x_i,y_i$, $i=1,\dots,n$, are given by
\begin{equation}
\label{e:Poisson bracket2}  
 \{ x_i, y_j \}= -\epsilon_i \delta_{i,j} 
  \quad (i,j=1,2,\dots,n),
\end{equation}
and all other brackets vanish. 
In terms of the complex coordinates $z_j = x_j + \ai y_j$, $j=1,2,\dots, n$, and their conjugates,
it follows from \eqref{e:Poisson bracket2} that the Poisson brackets among $z_j$ and $\bar z_j$ are given by
\begin{equation}
\label{e:Poisson bracket2 complex}
\{ z_i, \bar z_j \} = 2 \ai \epsilon_i \delta_{i,j},
   \quad 
 \{ z_i, z_j \} = \{ \bar z_i, \bar z_j \} = 0
\end{equation}
for $i,j = 1,2,\dots,n$.
In view of \eqref{e:Poisson bracket2 complex} we quantize $z_i$ and $\bar z_i$ by assigning
\begin{equation}
\label{e:canonical quantization2 complex}
 \widehat{z}_i = z_i, \quad
 \widehat{\bar z}_i = -2  \epsilon_i {\pd_{z_i}},
\end{equation}
so that they satisfy 
\begin{equation}
\label{e:commutation relation2 complex}
 [ \,\widehat{\vphantom{\bar z}z}_i, \,\widehat{\bar z}_j \,] = 2  \epsilon_i \delta_{i,j},
   \quad 
 [ \,\widehat z_i, \,\widehat z_j \,] = [ \,\widehat{\bar z}_i,\,\widehat{\bar z}_j \,] = 0
\end{equation}
for $i,j = 1,2,\dots,n$.
Therefore, we quantize the moment map $\mu$ and denote the quantized moment map by $\widehat \mu$ as follows:
\begin{align}
 \widehat \mu 
    &= -\frac{\ai}2 \begin{bmatrix} \,\widehat z_1 \\ \vdots \\ \,\widehat z_{n} \end{bmatrix}
                    \left( \, \widehat{\bar z}_1, \dots, \widehat{\bar z}_{n} \right) I_{p,q}
%        \notag  \\
    = \ai \begin{bmatrix} z_1 \\ \vdots \\ z_{n} \end{bmatrix}
          \left(  \pd_{z_1 }, \dots, \pd_{z_{n}} \right)
     = \ai z \tp{\pd_z}
        \label{e:quantized moment map2}
\end{align}
with $z = \tp{(z_1, \dots,z_{n})}$ and $\pd_z = \tp{(\pd_{z_1}, \dots,\pd_{z_{n}})}$.
Note that the quantization \eqref{e:canonical quantization2 complex} corresponds to taking a complex Lagrangian subspace $V'$ given by
\begin{equation}
\label{e:Lagrangian2}
 V':=\left\< \tfrac12(e_1 - \ai I e_1), \dots, \tfrac12(e_n - \ai I e_n) \right\>_\C  \subset W_\C,
\end{equation}
where $I$ denotes the complex structure on $W$ defined by $e_j \mapsto \ai e_j$, $\ai e_j \mapsto -e_j$ for $j=1,\dots,n$.
The classical observables $z_j=x_j + \ai y_j$ can be regarded as the coordinates on $V'$ with respect to this basis
under the identification $e_j \leftrightarrow \pd_{x_j}$ and $\ai e_j \leftrightarrow \pd_{y_j}$, $j=1,\dots,n$,
and $V'$ is naturally identified with $\C^n$.
Let $\PVprime$ denote the algebra of complex coefficient polynomial functions on $V$,
i.e., $\PVprime=\C[z_1,\dots,z_{n}]$,
and $\PD(V')$ the ring of polynomial coefficient differential operators on $V'$.
%%%%%%%%%%%%%%%%%%%%%%%%%%%%%%%%%%%%%%%%%%%%%%%%%%%%%%%%%%%%%%
%
%   Theorem: quantized moment map U(p,q) \times GL{1}-case 
%
%%%%%%%%%%%%%%%%%%%%%%%%%%%%%%%%%%%%%%%%%%%%%%%%%%%%%%%%%%%%%%
\begin{theorem}
\label{t:pi is a hom gl}
 For $X \in \g=\gl_{n}$, set $\pi(X)=\ai \<\widehat{\mu}, X\>$.
 Then the map
 \[
 \pi: \g \to \PD(V')
 \]
 is a Lie algebra homomorphism.
 In terms of the basis $\{ E_{i,j} \}$ for $\g$, it is given by
 \begin{equation}
 \label{e:explicit form of pi2}
 \pi(E_{i,j}) %= \ai \< \,\widehat \mu, E_{i,j} \> 
   = - z_j \pd_{z_i}
 \end{equation}
 for $i,j=1,\dots,n$.
\end{theorem}
\begin{proof}
The same argument as in Theorem \ref{t:pi is a hom sp} shows that $\pi$ is a Lie algebra homomorphism,
and \eqref{e:explicit form of pi2} follows immediately from \eqref{e:quantized moment map2}.
\end{proof}

It is clear from \eqref{e:explicit form of pi2} that $\pi(X) \in \PD(V')^{\GL{1}}$ for all $X \in \g$,
where $\GL{1}$ acts on $V'$ on the right.

\subsection{}%{Howe duality (U(p,q),U(k))}
Now let us consider $W^k$, the direct sum of $k$ copies of $W=(\C^n)_\R$, 
which is identified with the underlying real vector space of $\Mat{n \times k}(\C)$.
It is equipped with a symplectic form $\omega_k$ given by
\[
 \omega_k (z,w) = \impart \trace{z^* I_{p,q} w} \quad (z,w \in W^k),
\]
and is still acted on by $G=\U(p,q)$ symplectically by matrix multiplication on the left.
Under the identification of $e_{i,a} \leftrightarrow \pd_{x_{i,a}}$ and $\ai e_{i,a} \leftrightarrow \pd_{y_{i,a}}$,
we write an element of $W^k$ as $z ={}^t[z_1, \dots, z_n]$, 
where $z_i = x_i + \ai y_i$ are complex row vectors 
with $x_i=(x_{i,1},\dots,x_{n,k})$ and $y_i=(y_{i,1},\dots,y_{i,k})$ being real row vectors of size $k$ for $i=1,\dots,n$.
Then $\omega_k$ is given by
\begin{equation}
\label{e:symplectic form2_k}
 \omega_k 
%     = \sum_{1 \leqsl i \leqsl n} \epsilon_i \dd x_{i} \wedge \dd \! \tp y_{i} 
     =\sum_{1 \leqsl i \leqsl n, 1 \leqsl a \leqsl k} \epsilon_i \dd x_{i,a} \wedge \dd y_{i,a}
\end{equation}
at $z={}^t[z_1,\dots,z_n] \in W^k$.
Note that $\U(p,q)$ acts on $W$ on the left, while $\U(k)$ acts on it on the right and that both actions are symplectic.

%%%%%%%%%%%%%%%%%%%%%%%%%%%%%%%%%%%%%%%%%%%%%%%%%%%%%%%%%%%%%%
%
%   Propostion: moment map U(p,q) times U(k)
%
%%%%%%%%%%%%%%%%%%%%%%%%%%%%%%%%%%%%%%%%%%%%%%%%%%%%%%%%%%%%%%
\begin{proposition}
Let $(W^k,\omega_k)$ be the symplectic $G$-vector space as above.
Then the moment map $\mu:W^k \to \g_0^* \simeq \g_0$ is given by the same formula as \eqref{e:moment map2 closed} 
\[
 \mu = -\frac{\ai}2 z z^* I_{p,q}
\]
with $z \in W^k=\Mat{n \times k}(\C)$.
In particular, $\mu$ is $G$-equivariant and is $\U(k)$-invariant.
\end{proposition}
\begin{proof} 
As in the proof of Proposition \ref{p:moment map (Sp,O)},
if we regard $x_i, y_i, \pd_{x_i}$ and $\pd_{y_i}$ as row vectors and the products as the inner product on the space of row vectors,
then similar argument to Proposition \ref{p:moment map2} shows that the moment map $\mu:W^k \to \g_0$ is given by \eqref{e:moment map2 closed},
with the understanding that $z \in \Mat{n \times k}(\C)$.
The $\U(k)$-invariance is obvious, and the $G$-equivariance is verified as in Proposition \ref{p:moment map2}.
\end{proof}
It follows from \eqref{e:symplectic form2_k} that 
the Poisson brackets among the real coordinate functions $x_{i,a}, y_{i,a}$, $i=1,\dots,n; a=1,\dots,k$, are given by
\begin{equation}
\label{e:Poisson bracket2_k}  
 \{ x_{i,a}, y_{j,b} \}= -\epsilon_i \delta_{i,j} \delta_{a,b}
  \quad (i,j=1,\dots,n; a,b=1,\dots,k),
\end{equation}
and all other brackets vanish. 
It follows from \eqref{e:Poisson bracket2_k} that 
the Poisson brackets among the complex coordinates $z_{j,a} = x_{j,a} + \ai y_{j,a}$ and their conjugates are given by
\begin{equation}
\label{e:Poisson bracket2_k complex}
\{ z_{i,a}, \bar z_{j,b} \} = 2 \ai \epsilon_i \delta_{i,j} \delta_{a,b}
\end{equation}
for $i,j=1,\dots,n; a,b=1,\dots,k$, and all other brackets vanish. 
Therefore, we quantize $z_{i,a}$ and $\bar z_{i,a}$ by assigning
\begin{equation}
\label{e:canonical quantization2_k complex}
 \widehat{z}_{i,a} = z_{i,a}, \quad
 \widehat{\bar z}_{i,a} = -2  \epsilon_i {\pd_{z_{i,a}}},
\end{equation}
so that the nontrivial commutators are given by
\begin{equation}
\label{e:commutation relation2_k complex}
 [ \,\widehat{\vphantom{\bar z}z}_{i,a}, \,\widehat{\bar z}_{j,b} \,] = 2  \epsilon_i \delta_{i,j} \delta_{a,b}.
\end{equation}

Let $V'^k$ denote the direct sum of $k$ copies of $V'$, with $V'$ as in \eqref{e:Lagrangian2}.
Since $V'^k$ can be identified with $\Mat{n \times k}(\C)$, we write an element of $V'^k$ as $z=(z_{i,a})_{i=1,\dots,n;a=1,\dots,k}$.
Note then that $\GL{k}$ acts on $V'^k$ by matrix multiplication on the right, 
and hence acts on $\PVprimek$ by right regular representation, which we denote also by $\rho$ as in \eqref{e:right regular rep. of G'}.
Let $\PVprimek = \C[z_{i,a} ; i=1,\dots,n,a=1,\dots,k]$ be the algebra of complex polynomial functions on $V'^k$,
and $\PD(V'^k)$ the ring of polynomial coefficient differential operators on $V'^k$.

The quantized moment map $\widehat \mu$ is also given by the same formula as \eqref{e:quantized moment map2}:
\begin{equation*}
 \widehat \mu = \ai z \, \tp \pd_z.
\end{equation*}  
In this case, however, $z$ and $\pd_z$ are $n \times k$-matrices 
whose $(i,a)$-th entries are the multiplication operator $z_{i,a}$ and the differential operator $\pd_{z_{i,a}}$
for $i=1,\dots,n$ and $a=1,\dots,k$, respectively.

%%%%%%%%%%%%%%%%%%%%%%%%%%%%%%%%%%%%%%%%%%%%%%%%%%%%%%%%%%%%%%
%
%   Theorem: Lie algebra homo pi U(p,q) \times GL{k}-case 
%
%%%%%%%%%%%%%%%%%%%%%%%%%%%%%%%%%%%%%%%%%%%%%%%%%%%%%%%%%%%%%%
\begin{corollary}
\label{t:pi is a hom (gl,GL)}
 For $X \in \g=\gl_{n}$, set $\pi(X)=\ai \<\widehat{\mu}, X\>$.
 Then the map
 \[
 \pi: \g \to \PD(V'^k)
 \]
 is a Lie algebra homomorphism.
 In terms of the basis $\{ E_{i,j} \}$ for $\g$, it is given by
 \begin{equation}
 \label{e:explicit form of pi2_k}
 \pi(E_{i,j}) = - \sum_{a=1}^k z_{j,a} \pd_{z_{i,a}}
 \end{equation}
 for $i,j=1,\dots,n$.
 Moreover, $\pi(X)$ commutes with the action of the complex general linear group $\GL{k}$, 
 i.e., $\pi(X) \in \PD(V'^k)^{\GL{k}}$ for all $X \in \g$.
\end{corollary}
\begin{proof}
The first statement that $\pi$ is a Lie algebra homomorphism can be shown as in the proof of Theorem \ref{t:pi is a hom sp}.
It remains to show that $\widehat \mu$ commutes with the action of $\GL{k}$, 
which can be done in the following way.
By Lemma \ref{l:adjoint of rho(G')}, one obtains that
\begin{equation*}
 \mAd_{\rho(g)^{-1}} z = z g^{-1}   \quad \text{and} \quad
 \mAd_{\rho(g)^{-1}} \pd_z = \pd_z \tp g,
\end{equation*} 
from which it follows that
\[
 \mAd_{\rho(g)^{-1}}(z \tp \pd_z) 
     = ( \mAd_{\rho(g)^{-1}} z ) \tp{( \mAd_{\rho(g)^{-1}} \pd_z )} 
     = z g^{-1} g \tp \pd_z = z \tp \pd_z.
\]
This completes the proof.
\end{proof}
Similarly to the case of $\Sp(n,\R)$, 
it is well known that the irreducible decomposition of $\PVk$ under the joint action of $(\gl_n,\GL{k})$ is given by
\begin{equation}
% \PVprimek \simeq \sum_{ \substack{\sigma \in \widehat{\GL{}}_k \\ L(\sigma) \ne \{ 0 \} } } L(\sigma) \otimes V_\sigma,
 \PVprimek \simeq \sum_{ {\sigma \in \widehat{\GL{}}_k, \, L(\sigma) \ne \{ 0 \} } } L(\sigma) \otimes V_\sigma,
\end{equation}
where $V_\sigma$ is a representative of the class $\sigma \in \widehat{\GL{}}_k$, 
the set of all equivalence classes of the finite-dimensional irreducible representation of $\GL{k}$,
and $L(\sigma):=\operatorname{Hom}_{\,\GL{k}}(V_\sigma, \PVprimek)$ which is a finite-dimensional irreducible representation of $\gl_n$.
It is also well known that the action $\pi$ restricted to $\mathfrak k$ lifts to the complexification $K_\C$ of the maximal compact subgroup $K$ of $G=\U(p,q)$, 
which implies that $L(\sigma)$ is an irreducible $(\g,K_\C)$-module.

\section{Reductive dual pair $(\mathfrak{o}^*(2 n), \Sp_k)$}
\label{s:O-Sp}

In this section,
let $G$ denote the linear Lie group defined by
\begin{align*}
  \rmO^*(2 n) &= \{ g \in \U(n,n); \tp{g} S g = S \} \\
            &= \{ g \in \GL{2n}(\C); g^* I_{n,n} g = I_{n,n}, \tp{g} S g = S\},
\end{align*}
where $S$ denotes the nondegenerate symmetric matrix 
$\left[ \begin{smallmatrix}  & 1_n \\[2pt] 1_n &  \end{smallmatrix} \right]$
of size $2n \times 2n$.
Set $\g_0=\mathfrak{o}^*(2n)$, the Lie algebra of $G$, and take a basis for $\g_0$ as follows:
\begin{equation}
\label{e:basis for o*}
\begin{aligned}
 X_{i,j}^c &= E_{i,j}-E_{j,i}+E_{\ibar,\jbar}-E_{\jbar,\ibar}       &   &(1 \leqsl i < j \leqsl n), \\
 Y_{i,j}^c &= \ai (E_{i,j}+E_{j,i}-E_{\ibar,\jbar}-E_{\jbar,\ibar} )       &   &(1 \leqsl i \leqsl j \leqsl n), \\
 X_{i,j}^n &= E_{i,\jbar}-E_{j,\ibar}-E_{\ibar,j}+E_{\jbar,i}   &    & (1 \leqsl i < j \leqsl n), \\
 Y_{i,j}^n &= \ai (E_{i,\jbar}-E_{j,\ibar}+E_{\ibar,j}-E_{\jbar,i})   &    & (1 \leqsl i < j \leqsl n),
\end{aligned}
\end{equation}
where $E_{i,j}$ denotes the matrix unit of size $2 n \times 2 n$.
The complexified Lie algebra $\mathfrak{o}_{2n}$ of $\g_0=\mathfrak{o}^*(2n)$ is realized as
\begin{equation}
 \mathfrak{o}_{2n} = \{ X \in \Mat{2n}(\C); \tp{X} S + S X = O \}
\end{equation}
in this section, which we will denote by $\g$ below. 
It has a basis
\begin{equation}
\label{e:basis for o}
\begin{aligned}
 X_{i,j}^0 &= E_{i,j}-E_{\jbar,\ibar}   &    & (1 \leqsl i,j \leqsl n), \\
 X_{i,j}^+ &= E_{i,\jbar}-E_{j,\ibar}   &    & (1 \leqsl i < j \leqsl n), \\
 X_{i,j}^- &= E_{\jbar,i}-E_{\ibar,j}   &    & (1 \leqsl i < j \leqsl n).
\end{aligned}
\end{equation}

\subsection{}

Let $W=(\C^{2n})_\R$ and $\omega = \impart H$, where $H: \C^{2n} \times \C^{2n} \to \C$ is the Hermitian form given by
\[
 H(u,v) = u^* I_{n,n} v  \quad (u,v \in \C^{2n}).
\]
Namely, we consider the case we have discussed in \S \ref{s:GL-GL} with $p=q=n$. 
Note in particular that $\omega$ can be written as
\begin{equation}
\label{e:symplectic form3}
 \omega=\sum_{j=1}^{n} (\dd x_j \wedge \dd y_j - \dd x_{\jbar} \wedge \dd y_{\jbar} )
\end{equation}
at $v=x+\ai y \in W$ with $x=\tp{(x_1,\dots,x_{2n})},\, y=\tp{(y_1,\dots,y_{2n})} \in \R^{2n}$.
Then $(W,\omega)$ is a symplectic $G$-vector space, as above.

%%%%%%%%%%%%%%%%%%%%%%%%%%%%%%%%%%%%%%%%%%%%%%%%%%%%%%%%%%%%%%
%
%  Remarks on  quaternionic vector soace
%
%%%%%%%%%%%%%%%%%%%%%%%%%%%%%%%%%%%%%%%%%%%%%%%%%%%%%%%%%%%%%%
\begin{remarks}
\label{r:quaternionic vector space}
(i)\;
There is another realization of the Lie group $\rmO^*(2n)$ as a group consisting of the complex orthogonal matrices.
Namely,
\begin{equation*}
 \rmO^*(2n) = \{ g \in \GL{2n}; \tp g g=1, \tp g J_n g=J_n \};
\end{equation*}
we temporarily denote this realization of $\rmO^*(2n)$ by $G^\gamma$,
because the former realization $G$ is isomorphic to $G^\gamma$ by the correspondence
\(
 G \ni g \mapsto \gamma g \gamma^{-1} \in G^\gamma
\)
with $\gamma=\frac1{\sqrt 2}\left[\begin{smallmatrix} 1 & 1 \\[2pt] \ai & -\ai \end{smallmatrix}\right] \in \U(2n)$ (cf.~\cite{Helgason78}).

Let us consider the quaternionic vector space 
\[
 \H^n:=\left\{ v=\tp{(v_1,\dots,v_n)}; v_i \in \H \;(i=1,\dots,n) \right\},
\]
where $\H=\{ a + b \, \mathbf{i} + c \, \mathbf{j} + d \, \mathbf{k}; a,b,c,d \in \R \}$ 
denotes the skew-field of quaternions.
We regard $\H^n$ as a \textit{right} \H-vector space. 
If we identify $\mathbf{i} \in \H$ with $\ai \in \C$, then $\H^n$  is isomorphic to $\C^{2n}$ by the map
\begin{equation}
\label{e:iso H^n and C^{2n}}
\phi_1 : \H^n \to \C^{2n}, \quad 
  v = v' + \mathbf{j} \, v'' \mapsto \begin{bmatrix} v' \\ v'' \end{bmatrix}
     \qquad (v',v'' \in \C^n),
\end{equation}
which is in fact a $\C$-isomorphism.
Then $G^\gamma$ is characterized as the group consisting of $\H$-linear transformations on $\H^n$
that preserve the quaternionic skew-Hermitian form $C$ given by
\begin{equation}
 C(u,v):= u^* \mathbf{j}\, v \qquad  (u,v \in \H^n)
\end{equation}
(see \cite{GW10} for details).
%[and the relation $C$ and $H$...]

(ii)\;
There is another identification of $\H^n$ with a \C-vector space.
Namely, there is an isomorphism of $\H^n$ onto $\Mat{n \times 2}(\C)$ given by
\begin{equation}
\label{e:iso H^n and C^{n times 2}}
\phi_2 : \H^n \to  \Mat{n \times 2}(\C), \quad v = v' + v'' \mathbf{j} \mapsto \left[ v', \;  v'' \right].
\end{equation}
In this case, however, $\H^n$ is regarded as a {\em left} \H-vector space, 
and the map $\phi_2$ is a $\C$-isomorphism in this sense.
Since $\mathbf{j} \, v'' =\bar{v}'' \mathbf{j}$ for $v'' \in \C^n$, one sees that
\begin{equation}
 (\phi_2 \circ \phi_1^{-1}) (\begin{bmatrix} v' \\ v'' \end{bmatrix}) 
  = \left[ v', \;  \bar v'' \right].
\end{equation} 
Note that $\phi_2 \circ \phi_1^{-1}$ is an \R-isomorphism from $\C^{2n}$ onto $\Mat{n \times 2}(\C)$.

More generally, let us consider $(\H^n)^k$, the direct sum of $k$ copies of $\H^n$, which we regard as a left \H-vector space as above.
Then the multiplication on $(\H^n)^k$ on the right by an element of $\Mat{k}(\H)$, say, $a + b\, {\mathbf j}$ with $a,b \in \Mat{k}(\C)$, 
corresponds to the multiplication on $\Mat{n \times 2k}(\C)$ on the right by the complex $2k \times 2k$-matrix 
$\left[\begin{smallmatrix} a & b \\[2pt] -\bar b & \bar a \end{smallmatrix}\right]$.
%\begin{equation}
% \left[ v' \;\; v'' \right] \mapsto \left[v' \;\; v'' \right] g
%\end{equation}
\end{remarks}

%%%%%%%%%%%%%%%%%%%%%%%%%%%%%%%%%%%%%%%%%%%%%%%%%%%%%%%%%%%%%%
%
%  Lemma: vector fields generated by the basis o*
%
%%%%%%%%%%%%%%%%%%%%%%%%%%%%%%%%%%%%%%%%%%%%%%%%%%%%%%%%%%%%%%
\begin{lemma}
\label{l:VF o*}
The vector fields on $W$ generated by the basis \eqref{e:basis for o*} for $\g_0=\mathfrak{o}^*(2n)$ in the sense of \eqref{e:vector field} are given by
\begin{equation}
\label{e:VF o*}
\begin{aligned}
 (X_{i,j}^c)_W &= - x_j \pd_{x_i} - y_j \pd_{y_i} + x_i \pd_{x_j} + y_i \pd_{y_j}
                  - x_{\jbar} \pd_{x_{\ibar}} - y_{\jbar} \pd_{y_{\ibar}} + x_{\ibar} \pd_{x_{\jbar}} + y_{\ibar} \pd_{y_{\jbar}},  \\
 (Y_{i,j}^c)_W &=  y_j \pd_{x_i} + y_i \pd_{x_j} - y_{\jbar} \pd_{x_{\ibar}} - y_{\ibar} \pd_{x_{\jbar}}
                  - x_j \pd_{y_i} - x_i \pd_{y_j} + x_{\jbar} \pd_{y_{\ibar}} + x_{\ibar} \pd_{y_{\jbar}}, \\
 (X_{i,j}^n)_W &= - x_{\jbar} \pd_{x_i} + x_{\ibar} \pd_{x_j} + x_j \pd_{x_{\ibar}} - x_i \pd_{x_{\jbar}}
                   - y_{\jbar} \pd_{y_i} + y_{\ibar} \pd_{y_j} + y_j \pd_{y_{\ibar}} - y_i \pd_{y_{\jbar}},  \\
 (Y_{i,j}^n)_W &=  y_{\jbar} \pd_{x_i} - y_{\ibar} \pd_{x_j} + y_j \pd_{x_{\ibar}} - y_i \pd_{x_{\jbar}}
                 - x_{\jbar} \pd_{y_i} + x_{\ibar} \pd_{y_j} - x_j \pd_{y_{\ibar}} + x_i \pd_{y_{\jbar}}.
\end{aligned}
\end{equation}
\end{lemma}

For a given $v=\left[ \!\begin{smallmatrix} v' \\ v'' \end{smallmatrix} \!\right] \in \C^{2n}$ with $v',v'' \in \C^n$,
we set $v_+ := (\phi_2 \circ \phi_1^{-1}) (v)=[v',\; \bar v''] \in \Mat{n \times 2}(\C)$ for brevity.
By Remarks \ref{r:quaternionic vector space} (ii), $\Sp(1)$ acts on $W$ on the right via the $\R$-isomorphism $\phi_2 \circ \phi_1^{-1}$.

%%%%%%%%%%%%%%%%%%%%%%%%%%%%%%%%%%%%%%%%%%%%%%%%%%%%%%%%%%%%%%
%
%  Proposition: moment map O^*(2n)
%
%%%%%%%%%%%%%%%%%%%%%%%%%%%%%%%%%%%%%%%%%%%%%%%%%%%%%%%%%%%%%%
\begin{proposition}
\label{p:moment map3}
Let $(W,\omega)$ be as above and $G=\rmO^*(2n)$.
Then the moment map $\mu: W \to \g_0^* \simeq \g_0$ is given by
\begin{subequations}
\label{e:moment map3 closed}
\begin{align}
\mu(v) &= -\frac{\ai}2 \left( v \, v^* I_{n,n} - S \tp{(v \, v^* I_{n,n})} S \right)
\label{e:moment map3 closed column}           \\
       &= -\frac{\ai}2 
           \begin{bmatrix}
	      v_+ v_+^*  & - v_+ J_1 \tp v_+   \\
              - \bar v_+ J_1 v_+^* & - \bar v_+ \tp v_+
	   \end{bmatrix}
\label{e:moment map3 closed mat}
\end{align}
\end{subequations}
for $v=x + \ai y \in W$ with $x=\tp{(x_1,\dots,x_{2n})}, y=\tp{(y_1,\dots,y_{2n})} \in \R^{2n}$.
In particular, $\mu$ is $G$-equivariant and is $\Sp(1)$-invariant.
\end{proposition}
\begin{proof}
It follows from Lemma \ref{l:VF o*} that
\begin{align}
\label{e:moment map3 real}
 \< \mu, X \>
 &=\begin{cases}
    x_i y_j - x_j y_i - x_{\ibar} y_{\jbar} + x_{\jbar} y_{\ibar}    &\text{if} \quad X=X_{i,j}^c ; \\
    x_i x_j + x_{\ibar} x_{\jbar} + y_i y_j + y_{\ibar} y_{\jbar}    &\text{if} \quad X=Y_{i,j}^c ; \\
    x_i y_{\jbar} - x_j y_{\ibar} + x_{\ibar} y_j - x_{\jbar} y_i    &\text{if} \quad X=X_{i,j}^n ; \\
    x_i x_{\jbar} - x_j x_{\ibar} - y_j y_{\ibar} + y_i y_{\jbar}    &\text{if} \quad X=Y_{i,j}^n,
   \end{cases} 
     \\
\intertext{which can be rewritten in terms of complex coordinates defined by $z_i:=x_i + \ai y_i$, $i=1,\dots,2n$, and their complex conjugates as}
 \< \mu, X \>
 &=\begin{cases}
           -  \frac{\ai}2 ( \bar z_i z_j - \bar z_j z_i - \bar z_{\ibar} z_{\jbar} + \bar z_{\jbar} z_{\ibar} )   &\text{if} \quad X=X_{i,j}^c ; \\
 \hphantom{-} \frac12 ( \bar z_i z_j + \bar z_j z_i + \bar z_{\ibar} z_{\jbar} + \bar z_{\jbar} z_{\ibar} )       &\text{if} \quad X=Y_{i,j}^c ; \\
           -  \frac{\ai}2 ( \bar z_i z_{\jbar} - \bar z_j z_{\ibar} + \bar z_{\ibar} z_j - \bar z_{\jbar} z_i )   &\text{if} \quad X=X_{i,j}^n ; \\
 \hphantom{-} \frac12 ( \bar z_i z_{\jbar} - \bar z_j z_{\ibar} - \bar z_{\ibar} z_j + \bar z_{\jbar} z_i )       &\text{if} \quad X=Y_{i,j}^n.
   \end{cases} 
\end{align}
Thus, setting $v':=\tp{(z_1,\dots,z_n)}$ and $v'':=\tp{(z_{\bar 1},\dots,z_{\bar n})}$, one obtains that
\allowdisplaybreaks{
\begin{align*}
 \mu(v) &=\sum_{i < j} \<\mu, X_{i,j}^c \> (X_{i,j}^c)^\vee 
          + \sum_{i \leqsl j} \< \mu, Y_{i,j}^c \> (Y_{i,j}^c)^\vee 
          + \sum_{i < j} \< \mu, X_{i,j}^n \> (X_{i,j}^n)^\vee  
          + \sum_{i < j} \< \mu, Y_{i,j}^n \> (Y_{i,j}^n)^\vee  
                \\
     &= -\frac{\ai}2 \sum_{i,j=1}^n 
        \left(
                 (z_i \bar z_j + \bar z_{\ibar} z_{\jbar}) E_{i,j} - (\bar z_i z_j + z_{\ibar} \bar z_{\jbar}) E_{\ibar,\jbar}
                -(z_i \bar z_{\jbar} - \bar z_{\ibar} z_j) E_{i,\jbar} - (\bar z_i z_{\jbar} - z_{\ibar} \bar z_j) E_{\ibar,j}
        \right)
                \\
     &=-\frac{\ai}2 
           \begin{bmatrix}
              v' \tp{\bar v'} + \bar v'' \tp v'' &  -v' \tp{\bar v''} + \bar v'' \tp v' \\
              -\bar v' \tp v'' + v'' \tp{\bar v'} & - \bar v' \tp v' - v'' \tp{\bar v''}      
	   \end{bmatrix}
     = -\frac{\ai}2 \left( 
           \begin{bmatrix} v' \\ v'' \end{bmatrix} \tp{\left( \bar v', \;  - \bar v'' \right)}
         + \begin{bmatrix} \bar v'' \\ - \bar v' \end{bmatrix} \tp{\left( v'', \; - v' \right)}
                     \right)
               \\
     &= -\frac{\ai}2 \left( v \, v^* I_{n,n} - S \tp{(v \, v^* I_{n,n})} S \right).
\end{align*}
}%
Rewriting \eqref{e:moment map3 closed column} one obtains the second expression \eqref{e:moment map3 closed mat}.

The $\Sp(1)$-invariance of $\mu$ immediately follows from $\eqref{e:moment map3 closed mat}$, 
and the $G$-equivariance can be verified in the following way.
If $g \in G$, then
\[
 \mu(g v) = - \frac{\ai}2 \left( g v v^* g^* I_{n,n} - S \tp{( g v v^* g^* I_{n,n})} S \right)
     = - \frac{\ai}2 \left( g v v^* I_{n,n} g^{-1} -S \tp{(g v v^* I_{n,n} g^{-1})} S \right)
\]
since $g^* I_{n,n} = I_{n,n} g^{-1}$.
The second term in the brace of the right-hand side equals
\[
 S \tp{g^{-1}} \tp{(v v^* I_{n,n})} \tp g S = g S \tp{(v v^* I_{n,n})} S g^{-1}
\]
since $\tp g S = S g^{-1}$.
Thus,
\[
 \mu(g v) = -\frac{\ai}2 \left( g v v^* I_{n,n} g^{-1} - g S \tp{( v v^* I_{n,n})} S g^{-1} \right)
   =\Ad(g) \mu(v).
\]
This completes the proof.
\end{proof}

It follows from \eqref{e:symplectic form3} that the Poisson brackets among $x_i, y_i$, $i=1,\dots,2n$, are given by
\begin{equation}
 \label{e:Poisson bracket3 real}
 \{ x_i, y_j \} = - \delta_{i,j}, \quad \{ x_{\ibar}, y_{\jbar} \} = \delta_{i,j}
\end{equation}
for $i,j = 1,\dots,n$, and all other brackets vanish.
In terms of complex coordinates $z_j=x_j + \ai y_j$ for $j=1,\dots,2n$ and their conjugates,
it follows from \eqref{e:Poisson bracket3 real} that the Poisson brackets among them are given by
\begin{equation}
\label{e:Poisson bracket3 complex}
 \{ z_i, \bar z_j \} = \{ \bar z_{\ibar}, z_{\jbar} \} = 2 \ai \delta_{i,j}
\end{equation}
for $i,j=1,\dots,n$ and all other brackets vanish, as in \eqref{e:Poisson bracket2 complex}.
In view of \eqref{e:Poisson bracket3 complex}, we quantize them by assigning
\begin{equation}
\begin{aligned}
\label{e:canonical quantization3 complex}
 \widehat z_i &= z_i,  \quad &  \widehat{\bar z}_i &= -2 \pd_{z_i}, \\
 \widehat{\bar z}_{\ibar} &= \bar z_{\ibar}, \quad & \widehat z_{\ibar} &= -2 \pd_{\bar z_{\ibar}}
\end{aligned}
\end{equation}
for $i=1,\dots,n$ so that the nontrivial commutators among the quantized operators are given by
\begin{equation}
\label{e:commutation relation3 complex}
 [ \,\widehat{{\vphantom{\bar z}} z}_i, \, \widehat{\bar z}_j \, ] 
 = [ \,\widehat{\bar z}_{\ibar}, \, \widehat{{\vphantom{\bar z}} z}_{\jbar} \, ] = 2 \delta_{i,j}
\end{equation}
for $i,j=1,\dots,n$.

Let $I$ denote a complex structure on $W$ defined by $e_j \mapsto \ai e_j$ and $\ai e_j \mapsto -e_j$ for $j=1,\dots,2n$.
Under the identification $e_j \leftrightarrow \pd_{x_j}$ and $\ai e_j \leftrightarrow \pd_{y_j}$,
the classical observables $z_j$ and $\bar z_j$ introduced above can be regarded as the coordinate functions on $W_\C$ 
with respect to the basis $\frac12(e_j - \ai I e_j)$ and $\frac12(e_j + \ai I e_j)$ respectively for $j=1,\dots,2n$.
Note that $\bar z_j$ is no longer the complex conjugate of $z_j$ since $x_i$ and $y_i$ are now complex functions.
Then the quantization \eqref{e:canonical quantization3 complex} corresponds to taking a complex Lagrangian subspace $V$ given by
\begin{equation}
\label{e:Lagrangian3}
 V=\left\< \tfrac12(e_j - \ai I e_j), \tfrac12(e_{\jbar} + \ai I e_{\jbar}); j=1,\dots,n \right\>_\C.
\end{equation}
For simplicity, we set $w_j := \bar z_{\jbar}$, $j=1,\dots,n$, 
and write an element of $V=\Mat{n \times 2}(\C)$ as $[z,\: w]$ with $z=\tp{(z_1,\dots,z_n)}$ and $w=\tp{(w_1,\dots,w_n)}$ in what follows.

Now, we quantize the moment map $\mu$ according to \eqref{e:canonical quantization3 complex}
using its first expression \eqref{e:moment map3 closed column} and denote the quantized moment map by $\widehat \mu$ as follows:
\begin{subequations}
\label{e:quantized moment map3}
\begin{align}
 \widehat \mu 
    &= -\frac{\ai}2 \Biggl( \,
      \begin{bmatrix} \,\widehat z_1  \\ \vdots \\ \,\widehat z_{\bar n} \end{bmatrix}
          ( \, \widehat{\bar z}_1, \dots, \widehat{\bar z}_{\bar n} ) \, I_{n,n}
      - S I_{n,n} \begin{bmatrix} \,\widehat{\bar z}_1 \\ \vdots \\ \,\widehat{\bar z}_{\bar n} \end{bmatrix} 
          ( \, \widehat z_1, \dots, \widehat z_{\bar n} ) \, S
                 \,  \Biggr)
        \label{e:quantized moment map3 1st}
                    \\
    &= -\frac{\ai}2 \biggl( \,
       \begin{bmatrix}  z \\[3pt] -2 \pd_w  \end{bmatrix}
       \left( -2 \tp\pd_z, \tp w \right)  I_{n,n} 
       - S I_{n,n}
         \begin{bmatrix} -2 \pd_z \\[3pt] w  \end{bmatrix}
       \left( \tp z,  -2 \tp \pd_w \right) S
                   \, \biggr)
             \notag  \\[5pt]
    &= \ai \begin{bmatrix}
	    z \tp \pd_z + w \tp \pd_w & \frac12(z \tp w - w \tp z)  \\[3pt]
            2(\pd_z \tp \pd_w - \pd_w \tp \pd_z)  & -(\pd_w \tp w + \pd_z \tp z)
	   \end{bmatrix}
        \label{e:quantized moment map3 2nd}
 \end{align}
 where $z=\tp{(z_1,\dots,z_n)}, w=\tp{(w_1,\dots,w_n)}, \pd_z= \tp{(\pd_{z_1}, \dots,\pd_{z_{n}})}$
 and $\pd_w=\tp{(\pd_{w_1},\dots,\pd_{w_n})}$.
 The quantization of the second expression \eqref{e:moment map3 closed mat}, i.e., 
 \begin{equation}
 \label{e:quantized moment map3 3rd}
 \widehat \mu = -\frac{\ai}2 \begin{bmatrix}
			       \widehat{v}_+ \tp{\,\widehat{\bar v}_+} & - \widehat{v}_+ J_1 \tp{\,\widehat{v}_+} \\[3pt]
                              - \widehat{\bar v}_+ J_1 \tp{\,\widehat{\bar v}_+} & - \widehat{\bar v}_+ \tp{\,\widehat{v}_+}
			     \end{bmatrix}
\end{equation}
\end{subequations}
produces the same result as \eqref{e:quantized moment map3 2nd},
where \( \widehat v_+=[z,w] \) and \( \widehat{\bar v}_+=[-2 \pd_z, -2 \pd_w] \).

Let $\PV$ denote the algebra of complex coefficient polynomials on $V$,
i.e., $\PV=\C[z_1,\dots,z_n,w_1,\dots,w_n]$,
and $\PD(V)$ the ring of polynomial coefficient differential operators on $V$.
Note that the complex symplectic group of rank one
\[
 \Sp_1 = \{ g \in \GL{2}; \tp g J_1 g = J_1 \}
\]
acts on $V$ by matrix multiplication on the right,
and hence on $\PV$ by right regular representation, which we denote by $\rho$, as in \eqref{e:right regular rep. of G'}.
The right-action of $\Sp_1$ on $V$ coincides with the one on $\Mat{n \times 2}(\C)$ mentioned in Remarks \ref{r:quaternionic vector space} (ii).
%%%%%%%%%%%%%%%%%%%%%%%%%%%%%%%%%%%%%%%%%%%%%%%%%%%%%%%%%%%%%%
%
%  Theorem: pi is a Lie algebra hom o
%
%%%%%%%%%%%%%%%%%%%%%%%%%%%%%%%%%%%%%%%%%%%%%%%%%%%%%%%%%%%%%%
\begin{theorem}
\label{t:pi is a hom o}
 For $X \in \g=\mathfrak{o}_{2n}$, set $\pi(X) = \ai \<\widehat{\mu}, X\>$. %= \ai \widehat{\mu(X)}$.
 Then the map
 \[
 \pi: \g \to \PD(V)
 \]
 is a Lie algebra homomorphism.
 In terms of the basis \eqref{e:basis for o} for $\g$, it is given by 
 \begin{equation}
  \pi(X) = \begin{cases}
   - ( z_j \pd_{z_i} + w_j \pd_{w_i} + \delta_{i,j} )          &\text{if} \quad X=X_{i,j}^0 ; \\
   2 \, ( \pd_{z_i} \pd_{w_j} - \pd_{w_i} \pd_{z_j} )          &\text{if} \quad X=X_{i,j}^+ ; \\
   \frac12 \, ( z_j w_i - w_j z_i )                        &\text{if} \quad X=X_{i,j}^-.  %\\
           \end{cases}
 \end{equation}
 Moreover, $\pi(X)$ commutes with the action of $\Sp_1$ i.e., $\pi(X) \in \PD(V)^{\Sp_1}$ for all $X \in \g$.
\end{theorem}
\begin{proof}
It suffices to prove that $\pi(X)$ commutes with the right-action of $\Sp_1$.
For this, we use the second expression \eqref{e:quantized moment map3 3rd} of $\widehat \mu$.
It follows from Lemma \ref{l:adjoint of rho(G')} that
\begin{equation*}
 \mAd_{\rho(g)^{-1}} \widehat v_+ = \widehat v_+ g^{-1}   \quad \text{and} \quad
 \mAd_{\rho(g)^{-1}} \widehat{\bar v}_+ = \widehat{\bar v}_+ \tp g
\end{equation*}
for $g \in \GL{2}$. 
Therefore, if $g \in \Sp_1$ then one obtains
\allowdisplaybreaks{
\begin{align*}
 \mAd_{\rho(g)^{-1}} \widehat \mu 
   &= - \frac{\ai}2 \begin{bmatrix}
		     \widehat{v}_+ g^{-1} \tp{(\widehat{\bar v}_+ \tp g)} & -\widehat{v}_+ g^{-1} J_1 \tp{(\widehat{v}_+ g^{-1})}  \\[3pt]
		     - \widehat{\bar v}_+ \tp g J_1 \tp{(\widehat{\bar v}_+ \tp g)} & -\widehat{\bar v}_+ \tp g \tp{(\widehat{v}_+ g^{-1})} 
		    \end{bmatrix}
          \\
   &= - \frac{\ai}2  \begin{bmatrix}
		     \widehat{v}_+ \tp{\;\widehat{\bar v}_+} & -\widehat{v}_+ J_1 \tp{\;\widehat{v}_+}  \\[3pt]
		     - \widehat{\bar v}_+ J_1 \tp{\;\widehat{\bar v}_+} & -\widehat{\bar v}_+ \tp{\;\widehat{v}_+} 
		    \end{bmatrix}
     = \widehat \mu
\end{align*}
}%
since $\tp g J_1 g=J_1$.
This completes the proof.
\end{proof}

\subsection{}%{Howe duality (O_{2n}, Sp_k)}
Now let us consider $W^k$, the direct sum of $k$ copies of $W=(\C^{2n})_\R$, 
which we identify with $\Mat{2n \times k}(\C)$.
It is equipped with a symplectic form given by
\[
 \omega_k (u,v) = \impart \trace{u^* I_{n,n} v} \quad (u,v \in W^k),
\]
and is still acted on by $G=\rmO^*(2n)$ symplectically by matrix multiplication on the left.
Under the identification of $e_{i,a} \leftrightarrow \pd_{x_{i,a}}$ and $\ai e_{i,a} \leftrightarrow \pd_{y_{i,a}}$,
we write an element of $W^k$ as $v ={}^t[v_1, \dots, v_{2n}]$, 
where $v_i = x_i + \ai y_i$ are complex row vectors
with $x_i=(x_{i,1},\dots,x_{i,k})$ and $y_i=(y_{i,1},\dots,y_{i,k})$ being real row vectors of size $k$ for $i=1,\dots,2n$.
Then $\omega_k$ is given by
\begin{equation}
\label{e:symplectic form3_k}
 \omega_k 
     =\sum_{1 \leqsl i \leqsl n, 1 \leqsl a \leqsl k}  
          ( \dd x_{i,a} \wedge \dd y_{i,a} - \dd x_{\ibar,a} \wedge \dd y_{\ibar,a} )
\end{equation}
at $v={}^t[v_1,\dots,v_{2n}] \in \Mat{2n \times k}(\C)$.
Moreover, the isomorphisms $\phi_1$ and $\phi_2$ defined by \eqref{e:iso H^n and C^{2n}} and \eqref{e:iso H^n and C^{n times 2}} respectively
naturally extend to the one between $(\H^n)^k$ and $\Mat{2n \times k}(\C)$ and the one between $(\H^n)^k$ and $\Mat{n \times 2k}(\C)$ respectively, 
which we denote by the same symbols.
Then $\Sp(k)$ acts on $W^k$ on the right via the $\R$-isomorphism $\phi_2 \circ \phi_1^{-1}$, as above.

%%%%%%%%%%%%%%%%%%%%%%%%%%%%%%%%%%%%%%%%%%%%%%%%%%%%%%%%%%%%%%
%
%   Propostion: moment map in the case of O^*(2n) times Sp(k)
%
%%%%%%%%%%%%%%%%%%%%%%%%%%%%%%%%%%%%%%%%%%%%%%%%%%%%%%%%%%%%%%
\begin{proposition}
Let $(W^k,\omega_k)$ be the symplectic $G$-vector space as above.
Then the moment map $\mu:W^k \to \g_0^* \simeq \g_0$ is given by the same formulae as \eqref{e:moment map3 closed}. 
Namely, for $v={}^t[v', \; v''] \in W^k$ with $v',v'' \in \Mat{n \times k}(\C)$,
\begin{equation}
\label{e:moment map3 closed k-direct sum}
\begin{aligned}
\mu(v) &= -\frac{\ai}2 \left( v \, v^* I_{n,n} - S \tp{(v \, v^* I_{n,n})} S \right)
          \\
       &= -\frac{\ai}2 
           \begin{bmatrix}
	      v_+ v_+^*  & - v_+ J_k \tp v_+   \\
              - \bar v_+ J_k v_+^* & - \bar v_+ \tp v_+
	   \end{bmatrix},
\end{aligned}
\end{equation}
where $v_+ = (\phi_2 \circ \phi_1^{-1} )(v) \in \Mat{n \times 2k}(\C)$.
In particular, $\mu$ is $G$-equivariant and is $\Sp(k)$-invariant.
\end{proposition}
\begin{proof}
The vector fields on $W^k$ generated by the basis for $\g_0$ are given by the same formulae as \eqref{e:VF o*} in Lemma \ref{l:VF o*}, 
with the understanding that $x_i, y_i, \pd_{x_i}$ and $\pd_{y_i}$ are row vectors and the products stand for the inner product of row vectors.
Now, exactly the same argument as in Proposition \ref{p:moment map3} implies the proposition. 
\end{proof}

It follows from \eqref{e:symplectic form3_k} that 
the Poisson brackets among the real coordinate functions $x_{i,a}, y_{i,a}$ are given by
\begin{equation}
 \label{e:Poisson bracket3_k real}
 \{ x_{i,a}, y_{j,b} \} = - \delta_{i,j} \delta_{a,b}, 
     \quad 
 \{ x_{\ibar,a}, y_{\jbar,b} \} = \delta_{i,j} \delta_{a,b}
\end{equation}
and all other brackets vanish, 
and hence the nontrivial ones among the complex coordinate functions are given by
\begin{equation}
\label{e:Poisson bracket3_k complex}
 \{ z_{i,a}, \bar z_{j,b} \} = \{ \bar z_{\ibar,a}, z_{\jbar,b} \} = 2 \ai \delta_{i,j} \delta_{a,b}
\end{equation}
for $i,j=1,\dots,n$ and $a,b=1,\dots,k$.
Therefore, we quantize $z_{i,a}$ and $\bar z_{i,a}$ by assigning
\begin{equation}
\begin{aligned}
\label{e:canonical quantization3_k complex}
 \widehat z_{i,a} &= z_{i,a},  \quad &  \widehat{\bar z}_{i,a} &= -2 \pd_{z_{i,a}}, \\
 \widehat{\bar z}_{\ibar,a} &= \bar z_{\ibar,a}, \quad & \widehat z_{\ibar,a} &= -2 \pd_{\bar z_{\ibar,a}}
\end{aligned}
\end{equation}
so that the nontrivial commutators among the quantized operators are given by
\begin{equation}
\label{e:commutation relation3_k complex}
 [ \,\widehat{{\vphantom{\bar z}} z}_{i,a}, \, \widehat{\bar z}_{j,b} \, ] 
 = [ \,\widehat{\bar z}_{\ibar,a}, \, \widehat{{\vphantom{\bar z}} z}_{\jbar,b} \, ] = 2 \delta_{i,j} \delta_{a,b}
\end{equation}
for $i,j=1,\dots,n$ and $a,b=1,\dots,k$.

Let $V^k$ denote the direct sum of $k$ copies of $V$, with $V$ as in \eqref{e:Lagrangian3}.
Since $V^k$ can be identified with $\Mat{n \times 2k}(\C)$, 
we write an element of $V^k$ as $[z,w]$, where $z=(z_{i,a})$ and $w=(w_{i,a})$ are elements of $\Mat{n \times k}(\C)$,
and we set $w_{i,a}=\bar z_{\ibar,a}$ for $i=1,\dots,n$ and $a=1,\dots,k$ for simplicity, as above.
Let $\PVk = \C[z_{i,a}, w_{i,a} ; i=1,\dots,n,a=1,\dots,k]$ be the algebra of complex polynomial functions on $V^k$,
and $\PD(V^k)$ the ring of polynomial coefficient differential operators on $V^k$.
Then the complex symplectic group $\Sp_k$ acts on $V^k$ by matrix multiplication on the right,
and hence on $\PVk$ by right regular representation, which we denote by $\rho$, as usual.

The quantized moment map $\widehat \mu$ is given by the same formula as \eqref{e:quantized moment map3 2nd}:
\[
 \widehat \mu 
     = \ai \begin{bmatrix}
	    z \tp \pd_z + w \tp \pd_w & \frac12(z \tp w - w \tp z)  \\[3pt]
            2(\pd_z \tp \pd_w - \pd_w \tp \pd_z)  & -(\pd_w \tp w + \pd_z \tp z)
	   \end{bmatrix}.
\]
Here, $z$ (resp. $w$) and $\pd_z$ (resp. $\pd_w$) now denote $n \times k$-matrices 
whose $(i,a)$-th entries are the multiplication operator $z_{i,a}$ (resp. $w_{i,a}$) 
and the differential operator $\pd_{z_{i,a}}$ (resp. $\pd_{w_{i,a}}$) for $i=1,\dots,n$ and $a=1,\dots,k$.
%%%%%%%%%%%%%%%%%%%%%%%%%%%%%%%%%%%%%%%%%%%%%%%%%%%%%%%%%%%%%%
%
%   Theorem: O^*(2n) \times Sp_k-case 
%
%%%%%%%%%%%%%%%%%%%%%%%%%%%%%%%%%%%%%%%%%%%%%%%%%%%%%%%%%%%%%%
\begin{corollary}
\label{t:pi is a hom (o,Sp)}
 For $X \in \g=\mathfrak{o}_{2n}$, set $\pi(X)=\ai \<\widehat{\mu}, X\>$.
 Then the map
 \[
 \pi: \g \to \PD(V^k)
 \]
 is a Lie algebra homomorphism.
 In terms of the basis \eqref{e:basis for o} for $\g$, it is given by
\begin{equation}
\label{e:explicit form of pi3_k}
  \pi(X) = \begin{cases}
   - \sum_{a=1}^k ( z_{j,a} \pd_{z_{i,a}} + w_{j,a} \pd_{w_{i,a}} + k \delta_{i,j} )          &\text{if} \quad X=X_{i,j}^0 ; \\
   2 \sum_{a=1}^k \, ( \pd_{z_{i,a}} \pd_{w_{j,a}} - \pd_{w_{i,a}} \pd_{z_{j,a}} )          &\text{if} \quad X=X_{i,j}^+ ; \\
   \frac12 \sum_{a=1}^k \, ( z_{j,a} w_{i,a} - w_{j,a} z_{i,a} )                        &\text{if} \quad X=X_{i,j}^-.  %\\
           \end{cases}
\end{equation}
 Moreover, $\pi(X)$ commutes with the action of the complex symplectic group $\Sp_{k}$, 
 i.e., $\pi(X) \in \PD(V^k)^{\Sp_{k}}$ for all $X \in \g$.
\end{corollary}
\begin{proof}
The proof is essentially the same as that of Theorem \ref{t:pi is a hom o}. 
\end{proof}
Similarly to the cases discussed above,
it is well known that the irreducible decomposition of $\PVk$ under the joint action of $(\mathfrak{o}_{2n},\Sp_k)$ is given by
\begin{equation}
\label{e:irred_decomp_weil_rep3}
% \PVk \simeq \sum_{\substack{\sigma \in \widehat{\Sp}_k \\ L(\sigma) \ne \{ 0 \} } } L(\sigma) \otimes V_\sigma,
 \PVk \simeq \sum_{{\sigma \in \widehat{\Sp}_k, \, L(\sigma) \ne \{ 0 \} }} L(\sigma) \otimes V_\sigma,
\end{equation}
where $V_\sigma$ is a representative of the class $\sigma \in \widehat{\Sp}_k$, 
the set of all equivalence classes of the finite-dimensional irreducible representation of $\Sp_{k}$,
and $L(\sigma):=\operatorname{Hom}_{\,\Sp_{k}}(V_\sigma, \PVk)$ which is an infinite-dimensional irreducible representation of $\mathfrak{o}_{2n}$.
It is also well known that the action $\pi$ restricted to $\mathfrak k$ lifts to the complexification $K_\C$ of the maximal compact subgroup $K$ of $G=\rmO^*(2n)$, 
which implies that $L(\sigma)$ is an irreducible $(\g,K_\C)$-module.

\section{Lagrangian subspace}

%Concluding remarks are in order.
In this section, 
we take complex Lagrangian subspaces of $W_\C$ different from the ones considered in the previous sections in the cases where $G=\rmO^*(2n)$ and $\U(p,q)$,
and quantize the moment map to obtain finite-dimensional representations of $\o_{2n}$ and the oscillator representation of $\u(p,q)$.
Finally, we make an observation that the image of the Lagrangian subspace coincides with the associated variety of the corresponding irreducible $(\g,K_\C)$ (or $(\g,\tilde K_\C)$)-modules 
occurring in the irreducible decomposition of the space consisting of polynomial functions on the Lagrangian subspace under the joint action of $(\g,G')$.

\subsection{}
Let $G=\rmO^*(2n)$ and let $(W,\omega)$ be the symplectic $G$-vector space we discussed in \S \ref{s:O-Sp},
i.e., $W=(\C^{2n})_\R$ and $\omega$ is given by \eqref{e:symplectic form3}.
Let us now consider another complex Lagrangian subspace $V' \subset W_\C$ defined by
\begin{equation}
 V':=\left\< \tfrac12( e_1 - \ai I e_1), \dots, \tfrac12( e_{2n} - \ai I e_{2n}) \right\>_\C
\end{equation}
and the corresponding quantization
\begin{equation}
\label{e:canonical quantization3 complex finite-dim}
 \widehat{z}_i = z_i, \quad
 \widehat{\bar z}_i = -2  \epsilon_i {\pd_{z_i}}
\end{equation}
for $i=1,\dots,2n$ as in \S \ref{s:GL-GL}, which also satisfy \eqref{e:commutation relation3_k complex}.
Here $I$ denotes the complex structure on $W$ mentioned in \S \ref{s:O-Sp}.
Then the quantized moment map, which we denote by the same symbol $\widehat \mu$, is given by
\allowdisplaybreaks{
\begin{align}
 \widehat \mu 
    &= -\frac{\ai}2 \Biggl( \,
      \begin{bmatrix} \,\widehat z_1  \\ \vdots \\ \,\widehat z_{\bar n} \end{bmatrix}
          ( \, \widehat{\bar z}_1, \dots, \widehat{\bar z}_{\bar n} ) \, I_{n,n}
      - S I_{n,n} \begin{bmatrix} \,\widehat{\bar z}_1 \\ \vdots \\ \,\widehat{\bar z}_{\bar n} \end{bmatrix} 
          ( \, \widehat z_1, \dots, \widehat z_{\bar n} ) \, S
                 \,  \Biggr)
             \notag  \\
    &= -\ai \begin{bmatrix}
	     - z' \tp \pd_{z'} + \pd_{z''} \tp z''   &   -z' \tp \pd_{z''} + \pd_{z''} \tp z'  \\[3pt]
             - z'' \tp \pd_{z'} + \pd_{z'} \tp z''  &  -z'' \tp \pd_{z''} + \pd_{z'} \tp z'
	    \end{bmatrix}
        \label{e:quantized moment map3'}
\end{align}
}%
with $z'=\tp{(z_1,\dots,z_n)}, z''=\tp{(z_{\bar 1},\dots,z_{\bar n})}, \pd_{z'}= \tp{(\pd_{z_1}, \dots,\pd_{z_{n}})}$
and $\pd_{z''}=\tp{(\pd_{z_{\bar 1}},\dots,\pd_{z_{\bar n}})}$.
Therefore, 
in terms of the basis \eqref{e:basis for o} for $\g=\mathfrak{o}_{2n}$, $\pi(X):=\ai \<\widehat \mu, X\>$ is given by 
\begin{equation}
\label{e:explicit form of pi3_k prime}
  \pi(X) = \begin{cases}
   z_j \, \pd_{z_i} - z_{\ibar} \, \pd_{z_{\jbar}}          &\text{if} \quad X=X_{i,j}^0 ; \\
   z_{\jbar} \, \pd_{z_i} - z_{\ibar} \, \pd_{z_j}          &\text{if} \quad X=X_{i,j}^+ ; \\
   z_i \, \pd_{z_{\jbar}} - z_j \, \pd_{z_{\ibar}}            &\text{if} \quad X=X_{i,j}^-.  %\\
           \end{cases}
\end{equation}
Since each $\pi(X)$ preserves the degree of a homogeneous polynomial $f \in \mathscr{P}(V')=\C[z_1,\dots,z_{2n}]$ for $X \in \g$,
any irreducible representation occurring in the irreducible decomposition of $\mathscr{P}(V')$ is finite-dimensional.

\subsection{}
On the contrary, we will apply the quantization procedure introduced in \S \ref{s:O-Sp} to the case discussed in \S \ref{s:GL-GL}.
Namely, let $G=\U(p,q)$ and let $(W,\omega)$ be the symplectic $G$-vector space,
i.e., $W=(\C^{p+q})_\R$ and $\omega$ is given by \eqref{e:symplectic form2}. 
Now we quantize the complex coordinate functions $z_j=x_j + \ai y_j$ and $\bar z_j=x_j - \ai y_j$ in the following way 
(cf. \eqref{e:canonical quantization3 complex}):
\begin{equation}
\begin{aligned}
\label{e:canonical quantization2 complex weil rep}
 \widehat z_i &= z_i,  \quad &  \widehat{\bar z}_i &= -2 \pd_{z_i},    \qquad    &  &(i=1,\dots,p); \\
 \widehat{\bar z}_{\jbar} &= \bar z_{\jbar}, \quad & \widehat z_{\jbar} &= -2 \pd_{\bar z_{\jbar}}, \qquad    &  &(j=1,\dots,q),
\end{aligned}
\end{equation}
which also satisfy \eqref{e:commutation relation2 complex}.
This quantization corresponds to taking a complex Lagrangian subspace $V \subset W_\C$ defined by
\begin{equation}
\label{e:Lagrangian2 weil rep}
 V=\left\< \tfrac12(e_i - \ai I e_i), \tfrac12(e_{\jbar} + \ai I e_{\jbar}); i=1,\dots,p,j=1,\dots,q \right\>_\C,
\end{equation}
where $I$ denotes the complex structure on $W$ mentioned in \S \ref{s:GL-GL}.
For simplicity, we will write $w_j := \bar z_{\jbar}$, $j=1,\dots,q$, as in the previous section,
and write an element of $V$ as $\left[\begin{smallmatrix} z \\[2pt] w \end{smallmatrix}\right]$
with $z \in \C^p$ and $w \in \C^q$.
Then the quantized moment map, which we denote by the same symbol $\widehat \mu$, is given by
\begin{align}
 \widehat \mu 
    &= -\frac{\ai}2 \begin{bmatrix} \,\widehat z_1 \\ \vdots \\ \,\widehat z_{n} \end{bmatrix}
                    \left( \, \widehat{\bar z}_1, \dots, \widehat{\bar z}_{n} \right) I_{p,q}
%        \notag  \\
    = -\frac{\ai}2 \begin{bmatrix} z \\[3pt] -2 \pd_w \end{bmatrix}
          \left( -2 \tp\pd_z,  - \tp w \right)
         \notag \\
    &= \ai \begin{bmatrix} z \tp\pd_z & \frac12 z \tp w \\[3pt] -2 \pd_w \tp\pd_z & -\pd_w \tp w \end{bmatrix}
        \label{e:quantized moment map2'}
\end{align}
with $z=\tp{(z_1,\dots,z_p)}$, $\pd_z = \tp{(\pd_{z_1}, \dots,\pd_{z_p})}$, $w=\tp{(w_1, \dots, w_q)}$ and $\pd_w=\tp{(\pd_{w_1}, \dots, \pd_{w_q})}$.
In terms of the basis $\{ E_{i,j} \}$ for $\g=\gl_n$, $\pi(X):=\ai \<\widehat \mu, X\>$ is given by 
\begin{equation}
  \pi(X) = \begin{cases}
   -z_j \, \pd_{z_i}                  &\text{if} \quad X=E_{i,j} \quad (i,j=1,\dots,p); \\
   \hphantom{-}2 \pd_{z_i} \pd_{w_j}       &\text{if} \quad X=E_{i,\jbar} \quad (i=1,\dots,p;j=1,\dots,q) ; \\
   -\frac12 z_j w_i                   &\text{if} \quad X=E_{\ibar,j} \quad (i=1,\dots,q;j=1,\dots,p) ; \\
   \hphantom{-} \pd_{w_j} w_i            &\text{if} \quad X=E_{\ibar,\jbar} \quad (i,j=1,\dots,q).  %\\
           \end{cases}
\end{equation}

Let us now consider the $k$ direct sum $W_\C^k$, and its subspace $V^k$ with $V$ given in \eqref{e:Lagrangian2 weil rep} 
which is identified with $\Mat{n \times k}(\C)$. 
Then $\GL{k}$ acts on $V^k$ on the right by
\begin{equation}
\label{e:holomorphic action of GL_k}
%  \Mat{n \times k}(\C) \times \GL{k} \to \Mat{n \times k}(\C), \quad
 \begin{bmatrix} z \\ w \end{bmatrix} \mapsto \begin{bmatrix} z \, g \\[2pt] w \tp{g}^{-1} \end{bmatrix}
\end{equation}
for $g \in \GL{k}$, with $z = (z_{i,a}) \in \Mat{p \times k}(\C)$ and $w = (w_{j,a}) \in \Mat{q \times k}(\C)$,
and hence on $\mathscr{P}(V^k)$ by right regular representation, which we denote by $\rho$, as usual.
Note that \eqref{e:holomorphic action of GL_k} is the holomorphic extension of the standard right-action of $\U(k)$ 
on $\Mat{n \times k}(\C)$ given by $Z \mapsto Z \,g$ for $Z \in \Mat{n \times k}(\C)$ and $g \in \U(k)$.
Then, understanding that $z$ and $\pd_z$ (resp. $w$ and $\pd_w$) in \eqref{e:quantized moment map2'} stand for 
$p \times k$-matrices $(z_{i,a})$ and $(\pd_{z_{i,a}})$ (resp. $q \times k$-matrices $(w_{j,a})$ and $(\pd_{w_{j,a}})$) as in the previous sections, 
one obtains the following.
%%%%%%%%%%%%%%%%%%%%%%%%%%%%%%%%%%%%%%%%%%%%%
%
%  Theorem: oscillator rep. for U(p,q)
%
%%%%%%%%%%%%%%%%%%%%%%%%%%%%%%%%%%%%%%%%%%%%%
\begin{theorem}
For $X \in \g=\gl_n$, set $\pi(X):=\ai \<\widehat \mu, X\>$. 
Then the map
\[
 \pi: \g \to \PD(V^k)
\]
is a Lie algebra homomorphism.
Moreover, $\pi(X)$ commutes with the action of $\GL{k}$ on $V^k$, 
i.e., $\pi(X) \in \PD(V^k)^{\GL{k}}$ for all $X \in \g$.
\end{theorem}
\begin{proof}
We only show that $\pi(X) \in \PD(V^k)^{\GL{k}}$ for $X \in \g$.
It follows from Lemma \ref{l:adjoint of rho(G')} that
\begin{align*}
 \mAd_{\rho(g)^{-1}} z &= z \, g^{-1},  &  \mAd_{\rho(g)^{-1}} w &= w \tp g, \\
 \mAd_{\rho(g)^{-1}} \pd_z &= \pd_z \tp g,  &  \mAd_{\rho(g)^{-1}} \pd_w &= \pd_w \, g^{-1}
\end{align*}
for $g \in \GL{k}$.
Hence one obtains that
\begin{align*}
\mAd_{\rho(g)^{-1}} \widehat \mu 
  &= -\frac{\ai}2
     \begin{bmatrix} \mAd_{\rho(g)^{-1}} z \\[3pt] -2 \mAd_{\rho(g)^{-1}} \pd_w \end{bmatrix}
      \left[ -2 \tp{(\mAd_{\rho(g)^{-1}} \pd_z)}, \;  \tp{(\mAd_{\rho(g)^{-1}} w)} \right]
           \\
  &= -\frac{\ai}2
     \begin{bmatrix} z \, g^{-1} \\[3pt] -2 \, \pd_w \, g^{-1} \end{bmatrix}
      \left[ -2 g \tp{\pd_z}, \;  g \tp{w} \right]
           \\
  &= -\frac{\ai}2
     \begin{bmatrix} z \\[3pt] -2 \pd_w \end{bmatrix} g^{-1}
      g \left[ -2 \tp{\pd_z}, \;  \tp{w} \right]
%           \\
   = \widehat \mu.
\end{align*}
This completes the proof.  
\end{proof}
Therefore, the irreducible decomposition of $\PVk$ is given by
\begin{equation}
\label{e:irred_decomp_weil_rep2}
% \PVk \simeq \sum_{\substack{\sigma \in \widehat{\mathrm{GL}}_k \\ L(\sigma) \ne \{ 0 \} } } L(\sigma) \otimes V_\sigma,
 \PVk \simeq \sum_{{\sigma \in \widehat{\mathrm{GL}}_k, \, L(\sigma) \ne \{ 0 \} } } L(\sigma) \otimes V_\sigma,
\end{equation}
where $V_\sigma$ is a representative of the class $\sigma \in \widehat{\mathrm{GL}}_k$, 
the set of all equivalence classes of the finite-dimensional irreducible representation of $\GL{k}$,
and $L(\sigma):=\operatorname{Hom}_{\,\GL{k}}(V_\sigma, \PVk)$. 
It is well known that $L(\sigma)$ is an irreducible $(\g,K_\C)$-module of infinite dimension for any $\sigma \in \widehat{\GL{}}_k$ such that $L(\sigma) \ne \{ 0 \}$,
where $\g=\gl_{p+q}$ and $K_\C$ is the complexification of the maximal compact subgroup $K$ of $G=\U(p,q)$.

%
% relation between the moment map and the associated variety
%
\subsection{}
One can uniquely extend the moment map $\mu: W \to \g_0$ to the map from $W_\C$ into $\g$, which we denote by $\mu_\C$.
Then the images by $\mu_\C$ of the complex Lagrangian subspaces 
that have been considered in this and previous sections are all equal to the associated varieties of the corresponding representations,
which we will see below case by case.

%Let us first consider the case $G=\U(p,q)$, then $G=\rmO^*(2n)$, and finally $G=\Sp(n,\R)$.

\subsubsection{}%The cases $G=\U(p,q), \rmO^*(2n)$
First we consider the cases where $G=\U(p,q)$ and $\rmO^*(2n)$.
Let $K_\C$ be the complexification of the maximal compact group $K$ of $G$.
Then it is well known that $K_\C$ acts on $\p$ with the irreducible decomposition $\p=\p^+ \oplus \p^-$
and that the orbit space decomposition of $\p^+$ under $K_\C$ is given by
\begin{equation}
\label{e:K_C-orbit_decomp 2}
 \p^+ = \bigsqcup_{j=0}^{r} \mathscr{O}^{K_\C}_j 
        \qquad (r:=\Rrank G),
\end{equation}
where
\begin{align}
 \mathscr{O}^{K_\C}_j 
    & = \left\{ \begin{bmatrix} O & C \\ O & O \end{bmatrix} ;%
                 \begin{array}[c]{l} C \in \Mat{p \times q}(\C), \rank{C} = j  \end{array} 
        \right\}%
               &\text{for} \quad G=\U(p,q), 
          \\
%\intertext{and}
 \mathscr{O}^{K_\C}_j 
    & = \left\{ \begin{bmatrix} O & C \\ O & O \end{bmatrix} ;%
                 \begin{array}[c]{l}  C \in \Mat{n \times n}(\C), \tp{C}+C=O \\ \rank{C} = 2j \end{array} 
        \right\}%
               &\text{for} \quad G=\rmO^*(2n).
\end{align} 
Moreover, if we denote the closure of a orbit $\mathscr O$ by $\overline{\mathscr O}$, then
\begin{equation}
\label{e:K-orbit closure}
 \overline{\mathscr{O}^{K_\C}_j} = \bigsqcup_{i \leqsl j} {\mathscr{O}^{K_\C}_i}
%      = \left\{ \begin{bmatrix} O & C \\ O & O \end{bmatrix} \in \p ; \rank{C} \leqsl j \right\}
\end{equation}
(see \cite{KostantRallis71}).
Therefore, in view of the explicit formulae \eqref{e:moment map2 closed} and \eqref{e:moment map3 closed},
one finds that 
\begin{equation}
\label{e:associated_var 2 and 3}
 \mu_\C (V'^k)=\overline{\mathscr{O}^{K_\C}_0} = \{ 0 \}    
     \quad \text{and} \quad
 \mu_\C (V^k)= \overline{\mathscr{O}^{K_\C}_m}
\end{equation}
with $m=\min(k,r)$.
Since the associated varieties of the finite-dimensional representations 
and those of the irreducible representations $L(\sigma)$ occurring in \eqref{e:irred_decomp_weil_rep2} and \eqref{e:irred_decomp_weil_rep3}
are equal to $\{ 0 \}$ and $\overline{\mathscr{O}^{K_\C}_m}$ respectively (cf.~\cite{DES91}),
one concludes that the image of the complex Lagrangian subspace $V^k$ or $V'^k$ by $\mu_\C$
coincides with the associated variety corresponding to the irreducible representations occurring in $\PVk$ or in $\PVprimek$.

\subsubsection{}%The case $G=\Sp(n,\R)$}
In order to see this is the case for $G=\Sp(n,\R)$, 
we realize the symplectic group over $\R$ as its Cayley transform:
\begin{align*}
 G^\gamma :\hspace{-3pt}&=\{ \gamma g \gamma^{-1}; g \in G\} 
%     \\
           = \Sp_n \cap \U(n,n)  
%     \\
%           &= \left\{ g=\begin{bmatrix} a & b \\ \bar b & \bar a \end{bmatrix} \in \GL{2n}; \tp{g} J_n g = J_n \right\} 
\end{align*}
with $\gamma = \frac12 \left[ \begin{smallmatrix} 1 & 1 \\[2pt] -\ai & \ai \end{smallmatrix} \right]$.
In the rest of this subsection, however,
let us denote $G^\gamma$ just by $G$, and use the same symbols to denote the Cayley transforms of subgroups, Lie algebras etc.~ as those of the corresponding objects 
by abuse of notation if there is no risk of confusion.

One can also obtain the so-called Fock model of the oscillator representation by the canonical quantization of the moment map in the following way:
let us denote by $I$ the complex structure on $W=\R^{2n}$ defined by $e_i \mapsto e_{\ibar}$ and $e_{\ibar} \mapsto -e_i$ for $i=1,\dots,n$,
and introduce complex coordinates $z_i:=x_i + \ai y_i$ and their conjugates $\bar z_i:=x_i - \ai y_i$, $i=1,\dots,n$.
Namely, we regard $W=\R^{2n}$ as $(\C^n)_\R$; more precisely, let us define an $\R$-vector space $W_a$ by
\(
 W_a = \left\{ \left[\begin{smallmatrix} z \\[2pt] \bar z \end{smallmatrix}\right]; z \in \C^n \right\}
\)
and an $\R$-isomorphism from $W_a$ onto $W$ by 
\[
  \varphi_\gamma: W_a \to W, \qquad 
     \left[ \, \begin{matrix} z \\ \bar z \end{matrix} \, \right]
        \mapsto
     \gamma \left[ \, \begin{matrix} z \\ \bar z \end{matrix} \, \right]
    = \frac12 \begin{bmatrix} z + \bar z \\ -\ai (z - \bar z) \end{bmatrix}, 
\]
for $z=\tp{(z_1,\dots,z_n)}$ and $\bar z=\tp{(\bar z_1,\dots,\bar z_n)} \in \C^n$.
The Cayley transform $G$ acts on $W_a$ by $v \mapsto g v$ (matrix multiplication) for $g \in G$ and $v \in W_a$,
with respect to which $\varphi_\gamma$ is equivariant.
Moreover, one sees that $\varphi_\gamma^* \omega = \frac{\ai}2 \sum_{i=1}^n \dd z_i \wedge \dd \bar z_i$
and that the moment map $\mu:W_a \to \g_0$ is given by
\begin{equation}
\label{e:moment map1 complex}
  \mu(v) = \frac{\ai}2 v \tp{v} J_n
         = \frac{\ai}2 \begin{bmatrix} -z \tp{\bar z} & z \tp{z} \\ -\bar z \tp{\bar z} & \bar z \tp{z} \end{bmatrix}
\end{equation}
for $v = \tp{(z_1,\dots,z_n,\bar z_1,\dots, \bar z_n)} \in W_a$.

%%%%%%%%%%%%%%%%%%%%%%%%%%%%%%%%%%%%%%%%%%%%
%
%     remark on equivariance of mu
%
%%%%%%%%%%%%%%%%%%%%%%%%%%%%%%%%%%%%%%%%%%%%
\begin{remark}
If we temporarily distinguish the Cayley transform $\g_0^\gamma$ from $\g_0$ only in this remark,
it is easily verified that the following diagram is commutative:
\begin{equation}
\xymatrix{
   W_a \ar[r]^{\mu} \ar[d]_{\varphi_\gamma} & \g_0^\gamma  \\
   W \ar[r]_{\mu} & \g_0, \ar[u]_{\Ad(\gamma)}
}
\end{equation}
where the upper horizontal map denotes the moment map given by \eqref{e:moment map1 complex},
while the lower horizontal one by \eqref{e:moment map1 closed}.
\end{remark}

The Poisson brackets among $z_i$ and $\bar z_i$ are given by \eqref{e:Poisson bracket2 complex} with all $\epsilon_i=1$.
Therefore we quantize $z_i$ and $\bar z_i$ by assigning
\begin{equation}
\label{e:canonical quantization1 complex}
 \widehat{z}_i = z_i, \quad
 \widehat{\bar z}_i = -2 {\pd_{z_i}},
\end{equation}
so that they satisfy \eqref{e:commutation relation2 complex} with all $\epsilon_i=1$.
This quantization corresponds to the choice of the complex Lagrangian subspace $V$ of $W_\C=\C^{2n}$ given by
\begin{equation}
\label{e:lagrangian1 fock}
  V=\left\< \tfrac12(e_1 - \ai e_{\bar 1}),\dots,\tfrac12(e_n - \ai e_{\bar n}) \right\>_\C.
\end{equation}
Then the quantized moment map, which we denote by $\widehat \mu$ as always, is given by
\allowdisplaybreaks{
\begin{align}
 \widehat \mu 
    &= \frac{\ai}2 \begin{bmatrix} \,\widehat{z}_1 \\ \vdots \\ \,\widehat{\bar z}_{n} \end{bmatrix}
                    \left( \, \widehat{z}_1, \dots, \widehat{\bar z}_{n} \right) J_n
%        \notag  \\
    = \frac{\ai}2 \begin{bmatrix} z \\[3pt] -2 \pd_z \end{bmatrix}
          \left( \tp{z},  - 2\tp{\pd_z} \right) J_n
         \notag \\
    &= \ai \begin{bmatrix} z \tp\pd_z & \frac12 z \tp z \\[3pt] -2 \pd_z \tp\pd_z & -\pd_z \tp z \end{bmatrix}
        \label{e:quantized moment map1 fock}
\end{align}
}%
with $z=\tp{(z_1,\dots,z_n)}$, $\pd_z = \tp{(\pd_{z_1}, \dots,\pd_{z_n})}$.
In terms of the basis $\{ X^{\star}_{i,j} \}$ for $\g=\sp_n$, $\pi(X):=\ai \<\widehat \mu, X\>$ is given by 
\begin{equation}
  \pi(X) = \begin{cases}
   -\frac12 (z_j  \pd_{z_i} + \pd_{z_i} z_j)  &\text{if} \quad X=X_{i,j}^0; \\
   \hphantom{-} 2\pd_{z_i} \pd_{z_j}          &\text{if} \quad X=X_{i,j}^+ ; \\
   - \frac12 z_j z_i                          &\text{if} \quad X=X_{i,j}^-.
           \end{cases}
\end{equation}

Now let us take the $k$ direct sum $W_\C^k$ and its subspace $V^k$, with $V$ as in \eqref{e:lagrangian1 fock}.
When $V^k$ is identified with $\Mat{n \times k}(\C)$, $\GL{k}$ acts on $V^k$ on the right and hence on $\PVk$ by right regular representation.
Then, if one understands that $z$ and $\pd_z$ in \eqref{e:quantized moment map1 fock} stand for $n \times k$-matrices $(z_{i,a})$ and $(\pd_{z_{i,a}})$ respectively
and sets $\pi(X)=\ai \<\widehat \mu, X\>$ for $X \in \g=\sp_n$, one can show that the map
\( \pi: \g \to \PD(V^k) \)
is a Lie algebra homomorphism and that $\pi(X) \in \PD(V^k)^{\rmO_{k}}$ for all $X \in \g$.
The irreducible decomposition of $\PVk$ under the joint action of $(\sp_n,\rmO_k)$ is of course the same as \eqref{e:Howe duality1}.

It is known that the orbit space decomposition under $K_\C$ of $\p^+$ is given by the same formula as \eqref{e:K_C-orbit_decomp 2} with 
\[
 \mathscr{O}^{K_\C}_j 
     = \left\{ \begin{bmatrix} O & C \\ O & O \end{bmatrix} ;%
                 \begin{array}[c]{l}  C \in \Mat{n \times n}(\C), \tp{C}=C \\ \rank{C} = j \end{array} 
        \right\},
\]
and its closure $\overline{\mathscr{O}^{K_\C}_j}$ is given by the same formula as \eqref{e:K-orbit closure} (see \cite{KostantRallis71}).
Therefore, in view of \eqref{e:moment map1 complex}, one finds that
\begin{equation}
\label{e:associated_var 1}
 \mu_\C (V^k)= \overline{\mathscr{O}^{K_\C}_m}
\end{equation}
with $m=\min(k,r)$.
Hence the image of the complex Lagrangian subspace $V^k$ by $\mu_\C$
coincides with the associated variety corresponding to the irreducible representations occurring in $\PVk$, as in the previous cases.

%\begin{acknowledgements}
%If you'd like to thank anyone, place your comments here
%and remove the percent signs.
%\end{acknowledgements}

% BibTeX users please use one of
\bibliographystyle{amsalpha}
\bibliography{rep,geom}

\providecommand{\bysame}{\leavevmode\hbox to3em{\hrulefill}\thinspace}
\providecommand{\MR}{\relax\ifhmode\unskip\space\fi MR }
% \MRhref is called by the amsart/book/proc definition of \MR.
\providecommand{\MRhref}[2]{%
  \href{http://www.ams.org/mathscinet-getitem?mr=#1}{#2}
}
\providecommand{\href}[2]{#2}
\begin{thebibliography}{HKM{\O}12}

\bibitem[AY09]{AY09}
N.~Abe and H.~Yamashita, \emph{A note on {H}owe duality correspondence and
  isotropy representations for unitary lowest weight modules of {${\rm
  Mp}(n,{\bf R})$}}, J. Lie Theory \textbf{19} (2009), no.~4, 671--683.
  \MR{2599006 (2011f:22017)}

\bibitem[BW97]{BW97}
Sean Bates and Alan Weinstein, \emph{Lectures on the geometry of quantization},
  Berkeley Mathematics Lecture Notes, vol.~8, American Mathematical Society,
  Providence, RI; Berkeley Center for Pure and Applied Mathematics, Berkeley,
  CA, 1997. \MR{1806388 (2002f:53151)}

\bibitem[CG97]{CG97}
Neil Chriss and Victor Ginzburg, \emph{Representation theory and complex
  geometry}, Birkh\"auser Boston, Inc., Boston, MA, 1997. \MR{1433132
  (98i:22021)}

\bibitem[DES91]{DES91}
Mark~G. Davidson, Thomas~J. Enright, and Ronald~J. Stanke, \emph{Differential
  operators and highest weight representations}, Mem. Amer. Math. Soc.
  \textbf{94} (1991), no.~455, iv+102. \MR{1081660 (92c:22034)}

\bibitem[EHW83]{EHW83}
Thomas Enright, Roger Howe, and Nolan Wallach, \emph{A classification of
  unitary highest weight modules}, Representation theory of reductive groups
  ({P}ark {C}ity, {U}tah, 1982), Progr. Math., vol.~40, Birkh\"auser Boston,
  Boston, MA, 1983, pp.~97--143. \MR{733809 (86c:22028)}

\bibitem[GW10]{GW10}
R.~Goodman and N.~W. Wallach, \emph{{Symmetry, representations and
  invariants}}, Graduate Texts in Math., vol. 255, Springer Verlag, 2010.

\bibitem[Has11]{KinvDO}
T.~Hashimoto, \emph{{On the principal symbols of $K_{\mathbb C}$-invariant
  differential operators on Hermitian symmetric spaces}}, J. Math. Soc. Japan
  \textbf{63} (2011), 837--869, {arXiv:0804.4038 [math.RT]}.

\bibitem[Hel78]{Helgason78}
S.~Helgason, \emph{{Differential geometry, Lie groups, and symmetric spaces}},
  Pure and App. Math., vol.~80, Academic Press, 1978.

\bibitem[HKM14]{HKM14}
Joachim Hilgert, Toshiyuki Kobayashi, and Jan M{\"o}llers, \emph{Minimal
  representations via {B}essel operators}, J. Math. Soc. Japan \textbf{66}
  (2014), no.~2, 349--414. \MR{3201818}

\bibitem[HKM{\O}12]{HKMO12}
Joachim Hilgert, Toshiyuki Kobayashi, Jan M{\"o}llers, and Bent {\O}rsted,
  \emph{Fock model and {S}egal-{B}argmann transform for minimal representations
  of {H}ermitian {L}ie groups}, J. Funct. Anal. \textbf{263} (2012), no.~11,
  3492--3563. \MR{2984074}

\bibitem[How85]{Howe85}
Roger Howe, \emph{Dual pairs in physics: harmonic oscillators, photons,
  electrons, and singletons}, Applications of group theory in physics and
  mathematical physics ({C}hicago, 1982), Lectures in Appl. Math., vol.~21,
  Amer. Math. Soc., Providence, RI, 1985, pp.~179--207. \MR{789290 (86i:22036)}

\bibitem[How89a]{Howe_remarks}
\bysame, \emph{Remarks on classical invariant theory}, Trans. Amer. Math. Soc.
  \textbf{313} (1989), no.~2, 539--570. \MR{986027 (90h:22015a)}

\bibitem[How89b]{Howe_transcending}
\bysame, \emph{Transcending classical invariant theory}, J. Amer. Math. Soc.
  \textbf{2} (1989), no.~3, 535--552. \MR{985172 (90k:22016)}

\bibitem[HU91]{HU91}
R.~Howe and T.~Umeda, \emph{{The Capelli identity, the double commutant
  theorem, and multiplicity-free actions}}, Math. Ann. \textbf{290} (1991),
  565--619.

\bibitem[KR71]{KostantRallis71}
B.~Kostant and S.~Rallis, \emph{Orbits and representations associated with
  symmetric spaces}, Amer. J. Math. \textbf{93} (1971), 753--809. \MR{0311837
  (47 \#399)}

\bibitem[KV78]{KV78}
M.~Kashiwara and M.~Vergne, \emph{On the {S}egal-{S}hale-{W}eil representations
  and harmonic polynomials}, Invent. Math. \textbf{44} (1978), no.~1, 1--47.
  \MR{0463359 (57 \#3311)}

\bibitem[KV95]{Knapp-Vogan95}
A.~W. Knapp and D.~A. Vogan, \emph{{Cohomological induction and unitary
  representations}}, Princeton Mathematical Series, vol.~45, Princeton Univ.
  Press, 1995.

\bibitem[Woo91]{Woodhouse91}
N.~Woodhouse, \emph{{Geometric quantization}}, Oxford Univ. Press, 1991.

\end{thebibliography}

\nocite{AY09}
\nocite{HKM14}

\end{document}